\documentclass[10pt]{amsart}
\usepackage{amssymb, latexsym, amsmath,amscd, amsthm, amsgen, rotating}
\usepackage{dsfont}
\usepackage[mathscr]{eucal}
\usepackage[dvips,ps,all,rotate]{xy}
\xyoption{curve}
\xyoption{rotate}
\xyoption{frame}

\theoremstyle{plain}
\newtheorem{thm}{Theorem}[section]
\newtheorem{theorem}[thm]{Theorem}
\newtheorem{lemma}[thm]{Lemma}
\newtheorem{corollary}[thm]{Corollary}
\newtheorem{proposition}[thm]{Proposition}

\theoremstyle{definition}
\newtheorem{definition}[thm]{Definition}

\theoremstyle{remark}
\newtheorem{remark}[thm]{Remark}
\newtheorem{example}[thm]{Example}

\def\cO{{\mathcal{O}}}

\newcommand{\pbcirc}{\ensuremath{\bullet}} 
\newcommand{\cbcirc}{\ensuremath{\circ}}
\newcommand{\bcirc}{\ensuremath{\circledcirc}}

\newcommand{\Rkan}{{\text{R}}}

\newcommand{\catA}{{\mathcal{A}}}

\newcommand{\catC}{{\mathcal{C}}}

\def\Id{{\mathds{1}}}

\makeatletter
\def\rightharpoonupfill@{\arrowfill@\relbar\relbar\rightharpoonup}
\newcommand{\overrightharpoonup}{%
   \mathpalette{\overarrow@\rightharpoonupfill@}}
\makeatother

\begin{document}

\title{ Cooperads as Symmetric Sequences}

\author[B. Walter]{Benjamin Walter}
\address{
	Mathematics Research and Teaching Group\\
	Middle East Technical University \\ Northern Cyprus Campus \\
	Kalkanli, Güzelyurt, KKTC \\
	via Mersin 10, Turkey
}
\email{benjamin@metu.edu.tr}

\subjclass[2010]{18D50; 16T15, 17B62.}
\keywords{Cooperads, operads, coalgebras, Kan extensions}

\begin{abstract}
We give a brief overview of the basics of cooperad theory using a new definition
which lends itself to easy example creation and verification.
We also apply our definition to build the parenthesization and cosimplicial structures 
exhibited by cooperads and give examples.
\end{abstract}

\maketitle

\section{Introduction}

In the current work we discuss cooperads 
in generic symmetric monoidal categories from the point of view of symmetric sequences.
Fix a symmetric monoidal category $(\catC,\,\otimes)$.
Let us roughly recall the standard framework.  

Operads encode algebra structures.  The tautological example is the endomorphism operad
of an object $\textsc{end}(A) = \coprod_n \mathrm{Hom}(A^{\otimes n}, A)$.  
Operads have a natural grading
by levels expressing the ``arity'' of different ``operations'' (for example, 
$\textsc{end}(A)(n) = \mathrm{Hom}(A^{\otimes n}, A)$).  The symmetric 
group $\Sigma_n$ acts on the $n$-ary operations of an operad (for $\textsc{end}(A)(n)$
this action is by permutation of the $A^{\otimes n}$).  A graded object with 
$\Sigma_n$-actions is called a ``symmetric
sequence.''  Operads are further equipped with a 
composition product identifying the result of plugging operations into each other
(for example, 
$\textsc{end}(A) \circ \textsc{end}(A) \to \textsc{end}(A)$).
Very roughly, an operad is ``a bunch of objects with a rule for plugging them into each other''.

Operads encode algebra structures via maps of operads (preserving symmetric group actions
and composition structure).  So, for example, there is an operad 
$\textsc{lie}$ of formal Lie bracket expressions modulo Lie relations, 
along with a composition rule identifying the result of plugging
bracket expressions into each other.  A map of operads $\textsc{lie} \to \textsc{end}(A)$
identifies a specific endomorphism of $A$ for each formal Lie bracket expression.  This gives
$A$ the structure of a Lie algebra.

Coalgebra structures can also be defined via operads.  The coendomorphisms of an object
$\textsc{coend}(A) = \coprod_n \mathrm{Hom}(A,A^{\otimes n})$ also form an operad:
It is graded, with 
symmetric group action, and has a natural map 
$\textsc{coend}(A)\circ \textsc{coend}(A) \to \textsc{coend}(A)$ also given by
plugging things into each other.  Replacing $\textsc{end}$ by $\textsc{coend}$
changes algebra structures to coalgebra structures.
For example a map of operads $\textsc{lie} \to \textsc{coend}(A)$
identifies a coendomorphism of $A$ for each Lie bracket expression, thus giving $A$ 
a Lie coalgebra structure.

This is the point of view taken by \cite{Smit}, but there is an alternative.  For clarity, we 
will continue with the example of Lie algebras.  A Lie algebra structure is maps
$\textsc{lie}(n) \to \mathrm{Hom}(A^{\otimes n}, A)$ which is equivalent to maps 
$\textsc{lie}(n)\otimes A^{\otimes n} \to A$ (ignore $\Sigma_n$-actions for the moment).
Dually, a Lie coalgebra structure is maps $\textsc{lie}(n) \to \mathrm{Hom}(A,A^{\otimes n})$
which is equivalent to maps 
$\textsc{lie}(n)\otimes A \to A^{\otimes n}$ which is equivalent to 
$A \to \bigl(\textsc{lie}(n)\bigr)^* \otimes A^{\otimes n}$.  (Dualizing $\textsc{lie}(n)$
should not introduce trouble, because it is finite dimensional.)  The level-wise dual object
$\textsc{lie}\,\check{} = \coprod_n \bigl(\textsc{lie}(n)\bigr)^*$ has structure dual to that
of $\textsc{lie}$.  This is a cooperad.  (The precise definition is the subject of the current paper.)

Experience  \cite{SiWa1} \cite{SiWa2} has shown that it is sometimes more useful to directly
work with cooperads 
and cooperad structures when 
describing coalgebras rather than continually referring all the way back to operads and 
operad structures.  Also sometimes coalgebras can have a more natural expression 
as coalgebras over cooperads, rather than coalgebras over operads.
Just as operads can be thought of as ``a bunch of objects which are plugged into each other'',
cooperads can be thought of as 
``a bunch of objects where subobjects are contracted/quotiented''.

Unfortunately category theory causes a slight hitch when attempting to blindly dualize 
operad structure to define cooperads.  The dual of operad composition is cooperad 
cocomposition, which is similar except for some colimits being replaced by limits.  The 
problem comes when looking at associativity.  In a symmetric monoidal category $\otimes$
is left adjoint (to $\mathrm{Hom}$)  so it will commute with colimits.  This allows operad
composition products to be associative 
(e.g. $(\textsc{lie} \circ \textsc{lie}) \circ \textsc{lie} = 
\textsc{lie} \circ (\textsc{lie}\circ \textsc{lie})$).
However, this will generally not happen for cooperad cocomposition
(e.g. $(\textsc{lie}\,\check{} \pbcirc \textsc{lie}\,\check{}\,) \pbcirc \textsc{lie}\,\check{} \neq 
\textsc{lie}\,\check{} \pbcirc (\textsc{lie}\,\check{} \pbcirc \textsc{lie}\,\check{}\,)$).
This issue crops up for example, in the cooperadic cobar constructions of Ching in his
thesis~\cite{Chin1} and arXiv note~\cite{Chin2}. 

\medskip

We work by defining a new composition product -- a composition product of tree-functors.
The motivating intuition is that 
{\em the composition product of two symmetric sequences should
not itself be a symmetric sequence} -- in particular its group of symmetries is much too
large.  Maps to and from the tree-functor composition product can be expressed as maps to 
and from 
universal extensions, which yields the classical operad and cooperad composition products.
Using the tree-functor composition product (rather than its extension) when describing 
or defining cooperads greatly simplifies bookkeeping; though it turns out that, for operads,
it doesn't really make a difference.

We begin by introducing the notation of wreath product categories.  These are
inspired by the wreath product categories of Berger~\cite{Berg}, and at the most basic
level are merely Groethendieck constructions.  
Wreath product categories are defined so that they will 
be the natural source category of iterated composition products of symmetric sequences. 
We use this to give a simple definition of cooperads and prove all of the standard structure
holds.  Then we describe comodules and coalgebras.  We finish with simple examples 
related to work in \cite{SiWa1}, \cite{SiWa2}, and \cite{Wal1}.

In the sequel \cite{Wal2} we use the structure presented here to build cofree coalgebras,
connecting to the constructions of Fox~\cite{Fox1} and Smith~\cite{Smit}.

\subsection{Acknowledgments}
I would like to thank Dev Sinha, whose questions led to the inception of this work; as well
as Michael Ching who resolved many of my early confusions.  Also Clemens Berger, 
Bruno Vallette, and Jim McClure listened to early versions of these ideas and provided 
invaluable feedback.  Most of all, I must thank Kallel Sadok and the Mediterranean Institute 
for Mathematical Sciences (MIMS) for an invitation to speak at the conference on 
``Operads and Configuration
Spaces'' in June 2012, which led to me finally revising and clarifying these ideas which
have been on paper and bouncing around in my head for almost six years.
This work is based on the notes from my series of talks at MIMS.

\section{Wreath product categories}

This section is divided into two parts.  In the first subsection, we define wreath product 
categories using functors to the category of finite sets.
Our definition is related to, but more general than, the dual of refined partitions of 
sets as used in literature
by e.g. Arone-Mahowald \cite{ArMa}.  The salient difference between wreath categories
and refined partitions is that
wreath categories incorporate the empty-set (see Remark~\ref{rem:partition}).
In the second subsection, an equivalent definition is given in terms of labeled level trees --
a more familiar category for the discussion of operads.  

\subsection{Wreath Products}

Write $\Sigma_n$ for the category of $n$-element sets and set isomorphisms and
$\Sigma_\ast = \coprod_{n\ge 0} \Sigma_n$ for the category of all finite sets and set isomorphisms ($\Sigma_0 = \emptyset$).  
Our notation reflects the fact that a functor 
$\Sigma_n \to \catC$ is merely an object of $\catC$ with a $\Sigma_n$-action.

There is an alternative way to construct $\Sigma_n$.  
Write $\mathtt{FinSet}$ for the category of finite sets and \textbf{all} set maps, and
write $\mathtt{[n]}$ for the category 
$1 \xrightarrow{f_1} 2 \xrightarrow{f_2} \cdots  \xrightarrow{f_{n-1}} n$.  
Then $\Sigma_\ast$ is 
equivalent to the category of functors $\mathtt{[1]} \to \mathtt{FinSet}$ and natural
isomorphisms.  We generalize this to define wreath product categories.

\begin{definition}\label{def:wreath1}
$\Sigma_*^{\wr n}$ 
is the category of contravariant functors $\mathtt{[n]}\to\mathtt{FinSet}$ and natural isomorphisms.
\end{definition}

\begin{remark}
Objects of $\Sigma_\ast^{\wr n}$ are chains of morphisms in $\mathtt{FinSet}$, indexed in the following manner.
\[
S_1 \xleftarrow{f_1} S_2 \xleftarrow{f_2} \cdots \xleftarrow{f_{n-1}} S_n
\]
Since $\mathtt{[n]} \cong \mathtt{[n]}^{\mathrm{op}}$, the use of contravariant functors in Definition~\ref{def:wreath1} is purely cosmetic.  Using covariant functors would change nothing, except that indices would not line up as perfectly later on.
\end{remark}

Note that we are clearly defining the levels of a simplicial category.  Before continuing
in that direction, however, we will explain our choice of notation via an equivalent, hands-on definition of wreath products with a generic category $\catA$.

\begin{definition}\label{def:wreath2}
The wreath product category $\Sigma_n\wr \catA$ is 
the category with
\begin{itemize} 
\item objects $\mathrm{Obj}(\Sigma_n\wr\catA) = 
   \bigl\{ \{A_s\}_{s\in S} \big|\ S \in \mathrm{Obj}(\Sigma_n)\bigr\}$ 
given by $n$-element sets of decorated objects of $\catA$;
\item and morphisms 
$\bigl(\sigma;\ \{\phi_t\}_{t\in T}\bigr):\{A_t\}_{t\in T} \longrightarrow \{B_s\}_{s\in S}$ 
given by a set isomorphism $\sigma:T\to S$ and 
a set of $\catA$-morphisms $\phi_t:A_{t} \to B_{\sigma(t)}$.
\end{itemize}

The wreath product category $\Sigma_\ast \wr\catA$ is given by 
$\Sigma_\ast\wr\catA := \coprod_{n\ge 0}\Sigma_n\wr \catA$. 
\end{definition}

\begin{remark}
$\Sigma_0\wr \catA$ is the empty category, since $\{A_s\}_{s\in\emptyset} = \emptyset$.  Furthermore  $\Sigma_\ast \wr \Sigma_0 \cong \Sigma_\ast \cong \Sigma_1\wr\Sigma_\ast$.  These equivalences are given by writing objects of 
$\Sigma_\ast \wr \Sigma_0$ as $(S\leftarrow \emptyset)$ and objects of $\Sigma_1\wr \Sigma_\ast$ as $(\star \leftarrow S)$ and using the facts that $\emptyset$ is initial and a one point set $\star$ is final in $\mathtt{FinSet}$.  We make further use of these equivalences later.  Note that $\Sigma_\ast\wr\Sigma_1 \ncong\Sigma_\ast$ because one point sets are not initial in $\mathtt{FinSet}$.
\end{remark}

The following proposition is easy to check.

\begin{proposition}
Definitions \ref{def:wreath1} and \ref{def:wreath2} agree:  
\begin{itemize}
\item
 $\Sigma_\ast^{\wr 2} \cong \Sigma_\ast \wr \Sigma_\ast$, and more generally
\item
 $\Sigma_\ast^{\wr n} \cong \Sigma_\ast \wr \bigl(\Sigma_\ast^{\wr n-1}\bigr)
  \cong \overbrace{\Sigma_\ast \wr (\cdots \wr (\Sigma_\ast \wr \Sigma_\ast))}^n$.
\end{itemize}
\end{proposition}
\qed

Using notation from Definition~\ref{def:wreath2}, the endomorphisms of the wreath product
category $\Sigma_n\wr \Sigma_m$ correspond to the automorphisms of an
$n$ element set of $m$ element sets $S = \{A_1, \dots , A_n\}$ with $|A_i| = m$.
Elements within each $A_i$ can be permuted by $\Sigma_m$ and the $A_i$ ``blocks''
are permuted by $\Sigma_n$ -- this is the wreath product group 
$\Sigma_n\wr \Sigma_m$.
Thus, a functor 
$\Sigma_n\wr \Sigma_m \to \catC$ is an object of 
$\catC$ equipped with an action of the wreath product group $\Sigma_n \wr \Sigma_m$.
We view $\Sigma_\ast \wr \Sigma_\ast$ as a generalization of this basic example -- the 
``blocks'' $A_i$ no longer need to be same size, and there can be an arbitrary number of them.

We return to the simplicial structure.  Recall that there are standard ``face'' functors 
$\partial^n_i:\mathtt{[n]} \to \mathtt{[n-1]}$ for $1 \le i \le (n-1)$, 
given by composing morphisms or forgetting $1$ (for reasons to be explained shortly, 
we do not use the ``forget $n$'' face map, $\partial^n_n$).
\begin{align*} 
\partial^n_i\bigl(1\xrightarrow{f_{1}} \cdots \xrightarrow{f_{n-1}} n\bigr) &= 
  \bigl(1 \rightarrow \cdots \rightarrow (i-1) \xrightarrow{f_{i}\circ f_{i-1}} (i+1) 
       \to\cdots \to n\bigr) \\
\partial^n_1(1\xrightarrow{f_{1}} \cdots \xrightarrow{f_{n-1}}n\bigr) &=
  \bigl(2 \rightarrow \cdots \rightarrow n\bigr)
\end{align*}
 Furthermore,
(because we do not allow the use of $\partial^n_n$ functors) any chain of compositions  
$\partial^{2}_{i_{2}}\circ\cdots \circ\partial^n_{i_n}:\mathtt{[n]}\to\mathtt{[1]}$
 equals the functor
$\gamma^n:\mathtt{[n]} \to \mathtt{[1]}$ which forgets all but the top object.
\[ \gamma^n\bigl(1\xrightarrow{f_{1}} \cdots \xrightarrow{f_{n-1}} n\bigr) =
 \bigl( n \bigr) \]
We will write $\partial^n_i$ and $\gamma^n$ also for the induced functors 
$\partial^n_i:\Sigma_\ast^{\wr n} \to \Sigma_\ast^{\wr (n-1)}$, for $1\le i\le (n-1)$,
and 
$\gamma^n: \Sigma_\ast^{\wr n} \to \Sigma_\ast$.  When $n$ is clear from context
we may write merely $\partial_i$ and $\gamma$.

\begin{remark} 
In the notation of Definition~\ref{def:wreath2}, the map 
$\gamma^2 = \partial^2_1: \Sigma_\ast \wr \Sigma_\ast \to \Sigma_\ast$ is given by
$\{S_t\}_{t\in T} \mapsto \coprod_{T} S_t$.  All other $\partial^n_i$ and $\gamma^n$
are induced by this (see Proposition~\ref{prop:wreath_is_associative}).
\end{remark} 

Before describing the degeneracy maps, we explain the missing $\partial^n_n$.  Recall
that $\Sigma_\ast$ is equivalent to the full subcategory $\widetilde{\Sigma}_\ast = \Sigma_1\wr\Sigma_\ast \subset \Sigma_\ast\wr\Sigma_\ast$
of functors sending $1$ to a one element set.  More generally, $\Sigma_\ast^{\wr n}$
is equivalent to the full subcategory 
$\widetilde{\Sigma}_\ast^{\wr n} =\Sigma_1\wr \Sigma_\ast^{\wr n} \subset \Sigma_\ast^{\wr n+1}$
of functors sending $1$ to a one element set.  Under this correspondence 
the face functors 
$\tilde\partial^n_i:\widetilde{\Sigma}_\ast^{\wr n} \to 
    \widetilde{\Sigma}_\ast^{\wr n-1}$, for $1 \le i \le (n-1)$, 
are all given by composition; however the functor $\tilde\partial^n_n$ is not.
\[\tilde\partial^n_i\bigl(\star \xleftarrow{f_0} S_1 \xleftarrow{f_1} \cdots 
   \xleftarrow{f_{n-1}} S_n\bigr) = 
  \bigl(\star \leftarrow \cdots \leftarrow S_{i-1} \xleftarrow{f_{i-1}\circ f_i} S_{i+1} \leftarrow 
   \cdots \leftarrow S_n\bigr)
\]
(By convention, $\star = S_0$).
Our goal is to capture the structure of $\widetilde{\Sigma}_\ast^{\wr n}$ along with 
the composition maps $\tilde{\partial}^n_i$.  Instead of working with this directly,
we use the equivalent categories and functors 
$\Sigma_\ast^{\wr n}$ and $\partial^n_i$;  because in practice keeping track of the 
final, one point set at the bottom of each chain is unnecessarily tedious.

We return to the degeneracies, which are best written via the equivalent categories
$\widetilde{\Sigma}_\ast^{\wr n}$.  In this notation, the degeneracy functors
$\tilde s^n_i:\widetilde{\Sigma}_\ast^{\wr n} \to \widetilde{\Sigma}_\ast^{\wr n+1}$
for $0\le i \le n$ are the doubling maps.
\[ \tilde s^n_i\bigl(
\star\xleftarrow{f_0} S_1 \xleftarrow{f_1} \cdots \xleftarrow{f_{n-1}} S_n
  \bigr) = 
\bigl(\star \leftarrow \cdots \leftarrow S_i \xleftarrow{\mathrm{Id}} S_i \leftarrow \cdots \leftarrow S_n \bigr)
 \]
Note that defining the degeneracy $s^n_0$ on the level of $\Sigma_\ast^{\wr n}$
requires picking a distinguished one point set.  A reader averse to making choices should 
replace all $\Sigma_\ast$, $\partial_i$, etc. by $\widetilde{\Sigma}_\ast$, $\tilde\partial_i$, etc. from now on.

It is classical that the degeneracies $s^n_{i-1}$ and $s^n_i$ are each sections of the face map $\partial^{n+1}_i$. 
Thus face and degeneracy maps combine to give
a collection of
categories and functors:
$$\xymatrix@C=2pt{
\cdots \ar@<9pt>[rrrr] \ar@<3pt>[rrrr] \ar@<-3pt>[rrrr] \ar@<-9pt>[rrrr] 
    &&\quad&\ar@{-->}[ll] \ar@{-->}@<6pt>[ll] \ar@{-->}@<-6pt>[ll] 
            \ar@{-->}@<12pt>[ll] \ar@{-->}@<-12pt>[ll] & 
 \Sigma_\ast \wr\Sigma_\ast\wr\Sigma_\ast\wr\Sigma_\ast\ar@<6pt>[rrrr] \ar[rrrr] \ar@<-6pt>[rrrr] 
    &&\quad&\ar@{-->}@<3pt>[ll] \ar@{-->}@<-3pt>[ll] \ar@{-->}@<9pt>[ll] \ar@{-->}@<-9pt>[ll] &
  \Sigma_\ast\wr\Sigma_\ast\wr\Sigma_\ast \ar@<3pt>[rrrr] \ar@<-3pt>[rrrr] 
    &&\quad&\ar@{-->}[ll] \ar@{-->}@<6pt>[ll] \ar@{-->}@<-6pt>[ll] &
  \Sigma_\ast\wr\Sigma_\ast \ar[rrrr] 
    &&\quad&\ar@{-->}@<3pt>[ll] \ar@{-->}@<-3pt>[ll] &
  \Sigma_\ast
}$$
where the dashed, left-pointing arrows are sections of their neighboring right-pointing arrows
and all pairs of neighboring right-pointing arrows are coequalized by an arrow out of their
target.  Under the correspondence 
$\Sigma_\ast^{\wr n} \cong \widetilde{\Sigma}_\ast^{\wr n} 
  \subset \Sigma_\ast^{\wr n+1}$, this is very explicitly a simplicial category with
the bottom level as well as the first and last face maps removed; equivalently, an 
augmented simplicial category with two extra degeneracies.



\begin{remark}\label{rem:hat_empty1}
We could express {\bf all} of the standard face maps $\partial^n_i$, $1\le i\le n$, as 
compositions by writing $\Sigma_\ast^{\wr n} \cong \overline{\Sigma}_\ast^{\wr n} = \Sigma_1 \wr \Sigma_\ast^{\wr n} \wr \Sigma_0 \subset \Sigma_\ast^{\wr n+2}$, 
 the full subcategory of functors sending $(n+2)$ to the empty-set and $1$ to a one element set.  Then $\partial^n_n$ becomes:
\[
\overline{\partial}^n_n\bigl(
 \star \xleftarrow{f_0} S_1 \xleftarrow{f_1} \cdots \xleftarrow{f_{n-1}} S_n
 \xleftarrow{f_n} \emptyset \bigr) = \bigl(
 \star \leftarrow S_1 \leftarrow \cdots \leftarrow S_{n-1} \xleftarrow{f_{n-1}\circ f_n} \emptyset \bigr)
\]
The $\overline{\Sigma}_\ast^{\wr n}$  fit together to make an (unaugmented) simplicial category with two extra degeneracies.  In the next section, the levels of this will be given an alternate definition and called $\hat \emptyset_n$.  This structure is useful for constructing algebras and coalgebras instead of operads and cooperads.
\end{remark}

\begin{remark}\label{rem:partition}
Another construction which has been useful in the past for describing and working
with operads uses the category of sets equipped with iterated refinements of partitions
where morphisms are given by set isomorphisms respecting all partition equivalences 
(see Arone-Mahowald \cite{ArMa} and Ching \cite{Chin1}) .
A partition of a set $S$ is equivalent to a \textbf{surjective} set map $S \to T$ where
$T$ is the set of partitions.
An iterated partition of a set $S$ is equivalent to a functor from $\mathtt{[n]}$ to
the category of finite sets and \textbf{surjections} 
(instead of the category of finite sets and \textbf{all} set maps).
This is sufficient for describing operads and cooperads which are trivial in ``0-arity''.
So partitions cannot be used to describe, for example, 
an operad of algebras over an algebra.  Also  
missing 0-arity means that partitions cannot work with algebras (or coalgebras)
as just a special case of modules (or comodules).
\end{remark}

Before continuing with the next subsection, we will combine Definitions~\ref{def:wreath1} and \ref{def:wreath2} to get a more general definition of wreath products with generic categories, necessary to discuss associativity.

\begin{definition}\label{def:wreath3}
 The wreath product category $\Sigma_\ast^{\wr n} \wr \catA$ is the category with
\begin{itemize}
\item  $\mathrm{Obj}(\Sigma_\ast^{\wr n} \wr \catA) = 
\Bigl\{\bigl(F,\ \{A_s\}_{s\in F(n)}\bigr) \ | \  F \in \mathrm{Obj}(\Sigma_\ast^{\wr n}),\ A_s \in \mathrm{Obj}(\catA)\Bigr\}$ 
\item  morphisms 
$\bigl(\Phi;\ \{\phi_s\}_{s\in F(n)}\bigr): \bigl(F,\ \{A_s\}\bigr) \to \bigl(G,\ \{B_t\}\bigr)$ given by 
 a natural isomorphism $\Phi:F \to G$ and a set of $\catA$-morphisms 
$\phi_s:A_s \to B_{(\Phi n)(s)}$
\end{itemize}
\end{definition}

\begin{proposition}\label{prop:wreath_is_associative}
 Wreath product is associative.
\[ (\Sigma_\ast \wr \Sigma_\ast) \wr \Sigma_\ast \cong \Sigma_\ast \wr (\Sigma_\ast \wr \Sigma_\ast) \cong \Sigma_\ast^{\wr 3}\]
More generally we have the following.
\[ \Sigma_\ast^{\wr n} \wr \Sigma_\ast^{\wr m} \cong \Sigma_\ast^{\wr n+m}\]

Furthermore, the face maps $\partial^n_i$ are all induced by $\gamma^2 = \partial^2_1$.
\[ \partial^n_i = \mathrm{Id} \wr \gamma^2 \wr \mathrm{Id}\ :\  
   \Sigma_\ast^{\wr i-1} \wr \bigl(\Sigma_\ast \wr \Sigma_\ast\bigr) \wr \Sigma_\ast^{\wr n-i-1} \longrightarrow \Sigma_\ast^{\wr i-1} \wr \bigl( \Sigma_\ast\bigr) \wr \Sigma_\ast^{\wr n-i-1}
\]
\end{proposition}
\qed

For example $\partial^3_1 = \gamma^2 \wr \mathrm{Id}$ and 
$\partial^3_2 = \mathrm{Id} \wr \gamma^2$.

\subsection{Level trees}
In this subsection we connect the wreath product constructions of the previous subsection
with the standard, classical method of describing operads via trees.  

For our purposes
a tree is a (nonempty) 
non-cyclic, connected, finite 
graph whose vertices are distinguished as: 
a ``root vertex'' of valency 1,
a (possibly empty) set of ``leaf vertices'' of valency 1, and all other vertices called 
``interior vertices''.  We require each tree to have a root and at least one
interior vertex;  however, we do not require that interior vertices have valency $>1$ -- 
despite the oxymoron (in particular, we allow the tree with a root, an 
``interior vertex'' but no leaves as in Figure~\ref{F:2-trees}).  
A tree isomorphism is an isomorphism of vertex
and edge sets, preserving with root and leaf distinctions. 

For convenience of notation we will orient all edges of our trees so that they point 
towards the root vertex; when drawing trees, we will not explicitly indicate this 
orientation, but rather always position the root at the bottom and the leaves at the
top, with the understanding that all edges point downwards.
We will denote interior vertices with a darkened dot $\bullet$, but we will generally
not bother to draw the root or leaf vertices -- instead we will indicate only the edges
connecting to them.  Also for convenience, we will draw trees on the plane, however
we consider them as non-planar
objects.  In particular, we will not assert any planar orderings on vertices or edges. 

There is a natural height function on the vertices of trees -- assigning to each vertex
the number of vertices on the path between it and the root (the vertex adjacent to
the root has height 0; the root has height -1).  
A ``level $n$ tree'' is a tree whose leaves all have height $n$ and whose interior
vertices have height $< n$.  A ``level tree'' is a tree
which is level $n$ for some $n$.  Note that a level $n$ tree may have branches without
leaves which contain no interior vertices of height $(n-1)$, as in Figure~\ref{F:2-trees}.
In particular, a tree with no leaves may be level $n$ as well as level $(n+1)$, etc.

\begin{figure}[h]
$$ 
 \begin{aligned}\begin{xy}
  (0,0)*{\bullet};
  (0,-2)**\dir{-}
 \end{xy}\qquad  
 \begin{xy}
  (-2,4)*{\bullet}="1",
  (2,4)*{\bullet}="2",
  (0,0)*{\bullet}="M",
  "M";"1"**\dir{-},
  "M";"2"**\dir{-},
  "M";(0,-2)**\dir{-},
 \end{xy}\qquad \begin{xy}
  (0,4)*{\bullet}="1", 
  (0,7)="N1",
  (0,0)*{\bullet}="M",
  "M";"1"**\dir{-},
  "N1";"1"**\dir{-},
  "M";(0,-2)**\dir{-},
 \end{xy}\qquad
 \begin{xy}
  (-1,4)*{\bullet}="1", 
  (3,4)*{\bullet}="2",
  (-3,7)="1a",
  (1,7)="1b",
  (0,0)*{\bullet}="M",
  "M";"1"**\dir{-},
  "M";"2"**\dir{-},
  "1";"1a"**\dir{-},
  "1";"1b"**\dir{-},
  "M";(0,-2)**\dir{-},
 \end{xy}\qquad \begin{xy}
   (0, 0)*{\bullet}="M";
  (-6,4)*{\bullet}="1", 
  ( 0,4)*{\bullet}="2", 
  ( 6,4)*{\bullet}="3", 
  (-10,7)="1a",
  (-6,7)="1b",
  (-2,7)="1c",
  ( 4,7)="3a",
  ( 8,7)="3b",
  "M";(0,-2)**\dir{-},
  "M";"1"**\dir{-},
  "M";"2"**\dir{-},
  "M";"3"**\dir{-},
  "1a";"1"**\dir{-},
  "1b";"1"**\dir{-},
  "1c";"1"**\dir{-},
  "3a";"3"**\dir{-},
  "3b";"3"**\dir{-},
 \end{xy} 
 \end{aligned}\qquad\qquad\qquad
 \begin{aligned}\begin{xy}
  (0,0)*{\bullet}="R", 
  (0,4)*{\bullet}="c",
  (-6,4)*{\bullet}="l",
  (6,4)*{\bullet}="r",
  (-9,8)*{\bullet}="ll",
  (-3,8)*{\bullet}="lr",
  (6,8)*{\bullet}="rc",
  (-9,12)*{\bullet}="llc",
  (0,12)*{\bullet}="rcl", 
  (6,12)*{\bullet}="rcc", 
  (12,12)*{\bullet}="rcr", 
  (-12,15)="llcl",
  (-6,15)="llcr",
  (6,15)="rccc",
  "R";(0,-2)**\dir{-},
  "R";"l"**\dir{-},
  "R";"c"**\dir{-},
  "R";"r"**\dir{-},
  "l";"ll"**\dir{-},
  "l";"lr"**\dir{-},
  "r";"rc"**\dir{-},
  "ll";"llc"**\dir{-},
  "rc";"rcl"**\dir{-},
  "rc";"rcc"**\dir{-},
  "rc";"rcr"**\dir{-},
  "llc";"llcl"**\dir{-},
  "llc";"llcr"**\dir{-},
  "rcc";"rccc"**\dir{-},
 \end{xy}\end{aligned}
$$
\caption{Some examples of level 2 trees and a level 4 tree}\label{F:2-trees}
\end{figure}
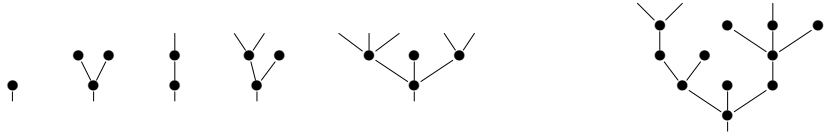

If $v$ is the target of the directed edge $e$ then we say $e$ is an ``incoming edge''
of $v$ and we write $\mathrm{In}(v)$ for the set of incoming edges of $v$.
In our drawings, incoming edges are edges abutting a vertex from above.
Each non-root vertex also has one ``outgoing edge'' (the abutting edge on 
the path from the vertex to the root), which will be drawn abutting the vertex
from below.

\begin{definition}
 A labeled level tree is a level tree equipped with labeling isomorphisms
 $\{l_v : S_v \xrightarrow{\ \cong\ } \mathrm{In}(v)\}_v$ from finite 
 sets to the sets of incoming edges at each vertex. 

 Let $\Psi$ be the category of all labeled level trees with morphisms given by 
 tree isomorphisms. 
 Let $\Psi_n$ be the full subcategory of $\Psi$ consisting of only level $n$ trees. 
\end{definition}

Since there is always only one incoming edge at the root, and never any incoming edges
at leaves, we may equivalently label only the incoming edges at interior vertices.

\begin{definition}
 Given a category $\catA$ define the wreath product category $\Psi\wr\catA$ to be 
 the category of all labeled level trees whose leaves are
 decorated by elements of $\catA$; morphisms are given by tree 
 isomorphisms equipped with $\catA$-morphisms between the leaf decorations compatible with
 the induced isomorphism of leaf sets.  Let $\Psi_n\wr\catA$ be the full subcategory
 of this consisting of only level $n$ trees. 
\end{definition}

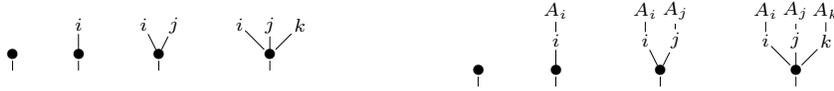
\begin{figure}[h]
$$
 \begin{aligned}\begin{xy}
  (0,0)*{\bullet};
  (0,-2)**\dir{-}
 \end{xy}\qquad  
 \begin{xy}
  (0,4)*+UR{\scriptstyle{i}}="1", 
  (0,0)*{\bullet}="M",
  "M";"1"**\dir{-},
  "M";(0,-2)**\dir{-},
 \end{xy}\qquad
 \begin{xy}
  (-2,4)*+UR{\scriptstyle{i}}="1", 
  (2,4)*+UR{\scriptstyle{j}}="2",
  (0,0)*{\bullet}="M",
  "M";"1"**\dir{-},
  "M";"2"**\dir{-},
  "M";(0,-2)**\dir{-},
 \end{xy}\qquad
 \begin{xy}
  (0,4)*+UR{\scriptstyle{j}}="1", 
  (-4,4)*+UR{\scriptstyle{i}}="2",
  (4,4)*+UR{\scriptstyle{k}}="3",
  (0,0)*{\bullet}="M",
  "M";"1"**\dir{-},
  "M";"2"**\dir{-},
  "M";"3"**\dir{-},
  "M";(0,-2)**\dir{-},
 \end{xy}\end{aligned}
\qquad\qquad\qquad 
 \begin{aligned}\begin{xy}
  (0,0)*{\bullet};
  (0,-2)**\dir{-}
 \end{xy}\qquad  
 \begin{xy}
  (0,8)*+UR{\scriptstyle{A_i}}="A1",
  (0,4)*+UR{\scriptstyle{i}}="1", 
  (0,0)*{\bullet}="M",
  "1";"A1"**\dir{-},
  "M";"1"**\dir{-},
  "M";(0,-2)**\dir{-},
 \end{xy}\qquad
 \begin{xy}
  (-2,8)*+UR{\scriptstyle{A_i}}="A1",
  (2,8)*+UR{\scriptstyle{A_j}}="A2",
  (-2,4)*+UR{\scriptstyle{i}}="1", 
  (2,4)*+UR{\scriptstyle{j}}="2",
  (0,0)*{\bullet}="M",
  "1";"A1"**\dir{-},
  "2";"A2"**\dir{-},
  "M";"1"**\dir{-},
  "M";"2"**\dir{-},
  "M";(0,-2)**\dir{-},
 \end{xy}\qquad
 \begin{xy}
  (-4,8)*+UR{\scriptstyle{A_i}}="A2",
  (0,8)*+UR{\scriptstyle{A_j}}="A1",
  (4,8)*+UR{\scriptstyle{A_k}}="A3",
  (0,4)*+UR{\scriptstyle{j}}="1", 
  (-4,4)*+UR{\scriptstyle{i}}="2",
  (4,4)*+UR{\scriptstyle{k}}="3",
  (0,0)*{\bullet}="M",
  "1";"A1"**\dir{-},
  "2";"A2"**\dir{-},
  "3";"A3"**\dir{-},
  "M";"1"**\dir{-},
  "M";"2"**\dir{-},
  "M";"3"**\dir{-},
  "M";(0,-2)**\dir{-},
 \end{xy}\end{aligned}
$$
\caption{Some objects of $\Psi_1$ and of $\Psi_1\wr \catA$}
\end{figure}

It is standard to note that the category $\Sigma_\ast$ may be identified with
the category $\Psi_1$ of labeled level 1 trees.
In this vein, the wreath product category $\Sigma_\ast \wr \catA$ 
may be identified with $\Psi_1\wr\catA$. 
More generally, the wreath product category 
$\Sigma_\ast \wr \Sigma_\ast$ is equivalent to 
the category $\Psi_2$ of all labeled level 2 trees; and the iterated 
wreath product category $\Sigma_\ast^{\wr n}$ is equivalent to $\Psi_n$ 
the category of all labeled level $n$ trees.

\begin{proposition}
 The following categories are equivalent.
\begin{itemize}
 \item $\Psi_1 \cong \Sigma_\ast$,
 \item $\Psi_1\wr \catA \cong \Sigma_\ast \wr \catA$
 \item $\Psi_n  \cong \Sigma_\ast^{\wr n}$
 \item $\Psi_n \wr \catA \cong \Sigma_\ast^{\wr n} \wr \catA$.
\end{itemize}
\end{proposition}
\qed

\begin{figure}[h]
$$ 
 \begin{aligned}\begin{xy}
  (0,0)*{\bullet};
  (0,-2)**\dir{-}
 \end{xy}\qquad  
 \begin{xy}
  (-2,4)*+UR{\scriptstyle{i_1}}="1", 
  (2,4)*+UR{\scriptstyle{i_2}}="2", 
  (-2,8)*{\bullet}="N1",
  (2,8)*{\bullet}="N2",
  (0,0)*{\bullet}="M",
  "M";"1"**\dir{-},
  "M";"2"**\dir{-},
  "N1";"1"**\dir{-},
  "N2";"2"**\dir{-},
  "M";(0,-2)**\dir{-},
 \end{xy}\qquad \begin{xy}
  (0, 4)*+UR{\scriptstyle{i}}="1", 
  (0, 8)*{\bullet}="N1",
  (0,12)*+UR{\scriptstyle{j}}="N1a",
  (0, 0)*{\bullet}="M",
  "M";"1"**\dir{-},
  "N1";"1"**\dir{-},
  "N1a";"N1"**\dir{-},
  "M";(0,-2)**\dir{-},
 \end{xy}\qquad
 \begin{xy}
  (-1,4)*+UR{\scriptstyle{i_1}}="1", 
  (3,4)*+UR{\scriptstyle{i_2}}="2",
  (-1,8)*{\bullet}="N1",
  (3,8)*{\bullet}="N2",
  (-3,12)*+UR{\scriptstyle{j_1}}="N1a",
  (1,12)*+UR{\scriptstyle{j_2}}="N1b",
  (0,0)*{\bullet}="M",
  "M";"1"**\dir{-},
  "M";"2"**\dir{-},
  "1";"N1"**\dir{-},
  "2";"N2"**\dir{-},
  "N1a";"N1"**\dir{-},
  "N1b";"N1"**\dir{-},
  "M";(0,-2)**\dir{-},
 \end{xy}\qquad \begin{xy}
   (0, 0)*{\bullet}="M";
  (-6,4)*+UR{\scriptstyle{i_1}}="1", 
  ( 0,4)*+UR{\scriptstyle{i_2}}="2", 
  ( 6,4)*+UR{\scriptstyle{i_3}}="3", 
  (-6,8)*{\bullet}="N1", 
  ( 0,8)*{\bullet}="N2", 
  ( 6,8)*{\bullet}="N3", 
  (-10,12)*+UR{\scriptstyle{j_1}}="N1a",
  (-6,12)*+UR{\scriptstyle{j_2}}="N1b",
  (-2,12)*+UR{\scriptstyle{j_3}}="N1c",
  ( 4,12)*+UR{\scriptstyle{k_1}}="N3a",
  ( 8,12)*+UR{\scriptstyle{k_2}}="N3b",
  "M";(0,-2)**\dir{-},
  "M";"1"**\dir{-},
  "M";"2"**\dir{-},
  "M";"3"**\dir{-},
  "1";"N1"**\dir{-},
  "2";"N2"**\dir{-},
  "3";"N3"**\dir{-},
  "N1a";"N1"**\dir{-},
  "N1b";"N1"**\dir{-},
  "N1c";"N1"**\dir{-},
  "N3a";"N3"**\dir{-},
  "N3b";"N3"**\dir{-},
 \end{xy} 
 \end{aligned}
$$
\caption{Some objects of $\Psi_2$}\label{fig:psi_2}
\end{figure}
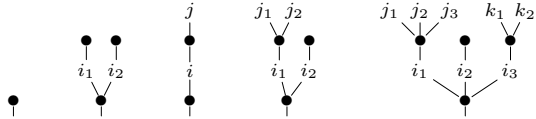

\begin{example}
The elements of $\widetilde{\Sigma}_\ast^{\wr 2}$ corresponding to the 
$\Psi_2$ elements in Figure~\ref{fig:psi_2} are given by the following chains of maps
in $\mathtt{FinSet}$.
\begin{itemize}
\item $\bigl(\star\leftarrow \emptyset \leftarrow \emptyset\bigr)$
\item $\bigl(\star\leftarrow \{i_2, i_2\} \leftarrow \emptyset\bigr)$
\item $\bigl(\star\leftarrow \{i\} \xleftarrow{f} \{j\} \bigr)$  where $f(j)=i$.
\item $\bigl(\star\leftarrow \{i_1, i_2\} \xleftarrow{f} \{j_1,j_2\} \bigr)$  where $f(j_s) = i_1$.
\item $\bigl(\star\leftarrow \{i_1, i_2, i_3\} \xleftarrow{f} \{j_1, j_2, j_3, k_1, k_2\}  \bigr)$
  where $f(j_s) = i_1$ and $f(k_t) = i_3$.
\end{itemize}
\end{example}

Under this identification, the functor $\gamma^2=\partial^2_1:\Psi_2 \to \Psi_1$
operates by forgetting the height $1$ vertices on a level $2$ tree.  Paths from
the height $0$ interior vertex to leaves (on level 2) are replaced by edges;  
the labeling of each such edge is 
given by the path labeling of the path which it replaces, as in Figure~\ref{fig:gamma_tree}.

\begin{figure}[h]
$$\gamma^2 = \partial^2_1\ : \ 
\begin{aligned}\begin{xy}
   (0, 0)*{\bullet}="M";
  (-6,4)*+UR{\scriptstyle{i_1}}="1", 
  ( 0,4)*+UR{\scriptstyle{i_2}}="2", 
  ( 6,4)*+UR{\scriptstyle{i_3}}="3", 
  (-6,8)*{\bullet}="N1", 
  ( 0,8)*{\bullet}="N2", 
  ( 6,8)*{\bullet}="N3", 
  (-10,12)*+UR{\scriptstyle{j_1}}="N1a",
  (-6,12)*+UR{\scriptstyle{j_2}}="N1b",
  (-2,12)*+UR{\scriptstyle{j_3}}="N1c",
  ( 4,12)*+UR{\scriptstyle{k_1}}="N3a",
  ( 8,12)*+UR{\scriptstyle{k_2}}="N3b",
  "M";(0,-2)**\dir{-},
  "M";"1"**\dir{-},
  "M";"2"**\dir{-},
  "M";"3"**\dir{-},
  "1";"N1"**\dir{-},
  "2";"N2"**\dir{-},
  "3";"N3"**\dir{-},
  "N1a";"N1"**\dir{-},
  "N1b";"N1"**\dir{-},
  "N1c";"N1"**\dir{-},
  "N3a";"N3"**\dir{-},
  "N3b";"N3"**\dir{-},
 \end{xy}\end{aligned}
 \ \ \longmapsto \ \ 
 \begin{aligned}\begin{xy}
   (0,0)*{\bullet}="M",
   (-14,8)*+UR{\scriptstyle{i_1j_1}}="N1a",
   (-7,8)*+UR{\scriptstyle{i_1j_2}}="N1b",
   (0,8)*+UR{\scriptstyle{i_1j_3}}="N1c",
   (7,8)*+UR{\scriptstyle{i_3k_1}}="N3a",
   (14,8)*+UR{\scriptstyle{i_3k_2}}="N3b",
   "M";(0,-3)**\dir{-},
   "M";"N1a"**\dir{-},
   "M";"N1b"**\dir{-},
   "M";"N1c"**\dir{-},
   "M";"N3a"**\dir{-},
   "M";"N3b"**\dir{-},
 \end{xy}\end{aligned}
$$
\caption{An example of $\gamma^2:\Psi_2 \to \Psi_1$}\label{fig:gamma_tree}
\end{figure}
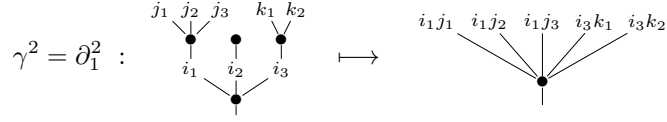

Similarly, the functors $\gamma^n:\Psi_n\to \Psi_1$ operate by forgetting all interior
vertices except for those of height 0; replacing paths by edges
carrying the paths' labels.  The face functors $\partial^n_i:\Psi_n \to \Psi_{n-1}$ for
$1 \le i \le n-1$ are given 
by forgetting only the vertices of level $i$ of a level $n$ tree.  
The disallowed face functor $\partial^n_n$ would forget the leaves.
The degeneracy
functors $s^n_i:\Psi_n \to \Psi_{n+1}$ for $0 \le i \le n$ are given by ``doubling'' -- replace each vertex $v$ at level $i$ by two vertices connected by a directed edge $e_v$, attached to the tree such that all incoming edges connect to source vertex of $e_v$ and the outgoing edge connects to the target vertex  (for the labeling, allow each edge to label itself $l_{t(e_v)}:\{e_v\} \to \{e_v\}$).  Note that the degeneracy $s^n_{n}$ doubles the leaf vertices -- the leaves of the resulting tree are the sources of the edges $e_v$.


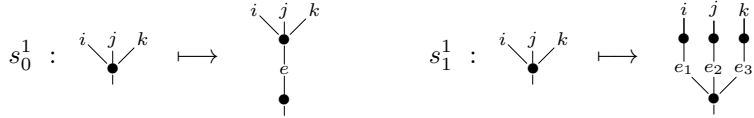
\begin{figure}[h]
\[ 
 s^1_0\ : \ 
\begin{aligned}\begin{xy}
  (0,4)*+UR{\scriptstyle{j}}="1", 
  (-4,4)*+UR{\scriptstyle{i}}="2",
  (4,4)*+UR{\scriptstyle{k}}="3",
  (0,0)*{\bullet}="M",
  "M";"1"**\dir{-},
  "M";"2"**\dir{-},
  "M";"3"**\dir{-},
  "M";(0,-2)**\dir{-},
 \end{xy}\end{aligned}
\ \ \longmapsto \ \ 
\begin{aligned}\begin{xy}
  (0,4)*+UR{\scriptstyle{j}}="1", 
  (-4,4)*+UR{\scriptstyle{i}}="2",
  (4,4)*+UR{\scriptstyle{k}}="3",
  (0,0)*{\bullet}="M",
  (0,-4)*+UR{\scriptstyle{e}}="U1",
  (0,-8)*{\bullet}="U2",
  "M";"1"**\dir{-},
  "M";"2"**\dir{-},
  "M";"3"**\dir{-},
  "M";"U1"**\dir{-},
  "U1";"U2"**\dir{-},
  "U2";(0,-10)**\dir{-}
 \end{xy}\end{aligned}
\qquad \qquad
 s^1_1\ :\  
\begin{aligned}\begin{xy}
  (0,4)*+UR{\scriptstyle{j}}="1", 
  (-4,4)*+UR{\scriptstyle{i}}="2",
  (4,4)*+UR{\scriptstyle{k}}="3",
  (0,0)*{\bullet}="M",
  "M";"1"**\dir{-},
  "M";"2"**\dir{-},
  "M";"3"**\dir{-},
  "M";(0,-2)**\dir{-},
 \end{xy}\end{aligned} 
\ \ \longmapsto \ \ 
\begin{aligned}\begin{xy}
  (0,12)*+UR{\scriptstyle{j}}="1t",
  (-4,12)*+UR{\scriptstyle{i}}="2t",
  (4,12)*+UR{\scriptstyle{k}}="3t",
  (0,8)*{\bullet}="1c",
  (-4,8)*{\bullet}="2c",
  (4,8)*{\bullet}="3c",
  (0,4)*+UR{\scriptstyle{e_2}}="1", 
  (-4,4)*+UR{\scriptstyle{e_1}}="2",
  (4,4)*+UR{\scriptstyle{e_3}}="3",
  (0,0)*{\bullet}="M",
  "M";(0,-2)**\dir{-},
  "1t";"1c"**\dir{-},  "1c":"1"**\dir{-},  "M";"1"**\dir{-},
  "2t";"2c"**\dir{-},  "2c":"2"**\dir{-},  "M";"2"**\dir{-},
  "3t";"3c"**\dir{-},  "3c":"3"**\dir{-},  "M";"3"**\dir{-}
 \end{xy}\end{aligned} 
\]
\caption{An example of $s^1_0,s^1_1:\Psi_1 \to \Psi_2$}\label{fig:degeneracy}
\end{figure}

\begin{remark}
 We very purposefully do not use the 
 notation $\Upsilon$ for our category of level trees, 
 since that notation is already commonly used to denote
 the category consisting of all trees.  The category $\Psi$ differs from this both
 on the level of objects (only level trees) and on the level of morphisms (only isomorphisms
 of trees -- in particular, no ``edge contraction'' maps).
\end{remark}

\begin{remark}\label{rem:hat_empty2}
 Note that $\Psi$ is not isomorphic to the category $\Psi_\ast = \coprod_n \Psi_n$.
 Write $\hat \emptyset_n$ for the full subcategory of $\Psi_n$ consisting of trees
 with no leaves.  Then $\hat \emptyset_n$ is a full subcategory of $\hat \emptyset_{n+1}$.
 In terms of the $\Psi_n$, the category $\Psi$ itself is given by
$$\Psi \cong \Psi_1 \coprod_{\hat \emptyset_1} \Psi_2 \coprod_{\hat \emptyset_2}
         \Psi_3 \coprod_{\hat \emptyset_3} \Psi_4 \cdots$$
In the notation of the previous subsection, an element of $\hat \emptyset_n$ is 
equivalent to a contravariant functor $\mathtt{[n]}\to \mathtt{FinSet}$ sending $n$ to the 
empty-set as in Remark~\ref{rem:hat_empty1}.
\end{remark}

\section{Symmetric sequences, composition products, and cooperads}

\subsection{Symmetric Sequences}

Let $(\catC, \otimes, 1_{\!\otimes\!})$ be a symmetric monoidal category with monoidal unit 
$1_{\!\otimes\!}$.
In order to have all desired Kan extensions exist, we will further require that
$\catC$ is cocomplete.   Write $\star_\catC$ for the final object of $\catC$.  
[In order to dualize to operads, we would require $\catC$ be complete with
initial object $\emptyset_\catC$.]  

\begin{definition}
A symmetric sequence is a functor $A:\Sigma_\ast \to \catC$.
\end{definition}

Recall that a functor $\Sigma_\ast \to \catC$ is equivalent to a sequence
of objects $\{A(n)\}_{n\ge 0}$ of $\catC$ along with a symmetric group action on each $A(n)$.
We will make use of this viewpoint when convenient without further comment. 
If $A$ is a symmetric sequence, then we will refer to $A(n)$ as the ``$n$-ary part of $A$''
since for operads it will encode $n$-ary algebra operations.  
(The ``$0$-ary operations" require no input.  For example, in the category of algebras over a
 field, elements of the base field are all $0$-ary operations.)

\subsection{Composition of Symmetric Sequences}

We define a ``product'' operation on symmetric sequences.  It is important
to note that our product will not itself be a symmetric sequence.
Instead it is a larger diagram, reflecting 
a larger group of symmetries.  The traditional composition product of operads as well as 
our cooperad composition product are Kan extensions of this symmetric sequence product.

\begin{definition}\label{D:composition product}
Given $A_1, \dots, A_n:\Sigma_\ast \to \catC$ define 
$(A_1\bcirc \cdots \bcirc A_n):\widetilde{\Sigma}_\ast^{\wr n} \to \catC$ by
\[
\bigl(\star \xleftarrow{f_0} S_1 \xleftarrow{f_1} \cdots \xleftarrow{f_{n-1}}S_n \bigr) 
   \longmapsto
\bigotimes_{0\le i \le n-1} \left( \bigotimes_{s \in S_{i}} A_{i+1}\bigl(f_{i}^{-1}(s)\bigr)\right)
\] 
with the convention that $\star = S_0$.

Define $A_1\pbcirc\cdots\pbcirc A_n$ to be the right Kan extension of $A_1\bcirc\cdots\bcirc A_n$ over the map 
$\gamma:\Sigma_\ast^{\wr n} \longrightarrow \Sigma_\ast$.
$$\xymatrixnocompile@R=20pt@C=80pt{ 
  \Sigma_\ast 
     \ar@/^/@{-->}[dr]^{\qquad \quad A_1\pbcirc \cdots \pbcirc A_n\ =\ \Rkan_{\gamma}A_1\bcirc \cdots\bcirc A_n } & 
\\
 \Sigma_\ast \wr \cdots\wr \Sigma_\ast  
    \ar[r]_(.55){A_1\bcirc\cdots\bcirc A_n}="b" \ar[u]^{\gamma} &
 \catC
\ar@{=>}"b"+<0pt,25pt>;"b"_(.3){\iota}
}$$
Write $\iota:(A_1\pbcirc\cdots\pbcirc A_n)\,\gamma \to A_1\bcirc\cdots\bcirc A_n$ for the universal natural transformation.

[Dually, define $A_1\cbcirc\cdots \cbcirc A_n$ to be the left Kan extension over $\gamma$.]
\end{definition}

Using the notation of Definition~\ref{def:wreath3}, we can generalize the above 
definition slightly in order to discuss associativity.

\begin{definition}
Given $A:\Sigma_\ast^{\wr n} \rightarrow \catC$ and $B:\catA\rightarrow \catC$, define 
$(A\bcirc B):\Sigma_\ast^{\wr n} \wr \catA \longrightarrow \catC$ by 
$$\bigl(A\bcirc B\bigr) \Bigl(F,\ \{A_s\}_{s\in F(n)}\Bigr) = 
    A(F) \otimes \left( \bigotimes_{s\in F(n)} B(A_s) \right).$$
\end{definition}

Short calculations yield the following propositions.

\begin{proposition}
 The operation $\bcirc$ is associative.
 \[ 
  (A_1 \bcirc A_2) \bcirc A_3 \ \cong \ 
  A_1 \bcirc A_2 \bcirc A_3\  \cong \ 
  A_1 \bcirc (A_2 \bcirc A_3)   
\]
\end{proposition}
\qed

\begin{proposition}
Given $A$, $B$ symmetric sequences, $A\pbcirc B$ is given by
\[
 (A\pbcirc B)(n) \ =  \ 
    \prod_{k\ge 0} \,\left(\, \prod_{\sum r_{\!_i}\,=\,n} \!\!\!
        A(k)\otimes B(r_1)\otimes \cdots \otimes B(r_k)
        \right)^{_{\displaystyle \Sigma_k}} 
\]
\end{proposition}
\qed

Note that $\pbcirc$ is probably not associative.  This will be discussed in greater
detail in the next section (see Proposition~\ref{L:parenthesizations}).  The operation
$\bcirc$ is clearly functorial.  If $F:A_1 \to A_2$ and $G:B_1\to B_2$ are natural
transformations of functors $A_1,A_2:\Sigma_\ast^{\wr n} \to \catC$ and
$B_1,B_2:\Sigma_\ast^{\wr m} \to \catC$, then we write 
$(F\bcirc G): (A_1 \bcirc B_1) \to (A_2 \bcirc B_2)$ for the induced natural
transformation of functors 
$\Sigma_\ast^{\wr n}\wr\Sigma_\ast^{\wr m} \to \catC$.

\subsection{Cocomposition and Coface Maps}

\begin{definition}\label{def:seq_with_cocomposition}
A symmetric sequence with cocomposition is 
$(A,\ \tilde\Delta)$ where $\tilde\Delta$ is a   
cocomposition natural transformation 
$\tilde\Delta:A\,\gamma^2 \longrightarrow A\bcirc A$
of functors $\Sigma_\ast\wr \Sigma_\ast \to \catC$ compatible with the face maps
$\partial^3_1 = (\gamma^2 \wr \mathrm{Id})$ and 
$\partial^3_2 = (\mathrm{Id} \wr \gamma^2)$.

Write $\Delta$ for the associated universal natural transformation of symmetric sequences 
  $\Delta:A\longrightarrow A\pbcirc A$.
\end{definition}

In other words, the following diagram of functors
$\Sigma_\ast\wr \Sigma_\ast \wr\Sigma_\ast \to \catC$ should commute.
\begin{equation} \label{eqn:cocomp}\begin{aligned}
\xymatrixnocompile@R=2pt{
& (A\bcirc A) (\gamma^2\wr \text{Id}) \ar@/^/[dr]^(.55){\tilde\Delta\wr\text{Id}} & \\
A \gamma^3 
   \ar@/^/[ur]^(.35){\tilde\Delta} \ar@/_/[dr]_(.35){\tilde\Delta} & &
  \hskip -10pt A\bcirc A\bcirc A \\
& (A\bcirc A) (\text{Id}\wr \gamma^2) \ar@/_/[ur]_(.55){\text{Id}\wr\tilde\Delta} & 
}\end{aligned}\end{equation}
The upper path uses the factorization
$\gamma^3 = \partial^2_1 \circ \partial^3_1 = \gamma^2 \circ (\gamma^2 \wr \mathrm{Id})$ and the lower path uses the factorization
$\gamma^3 = \partial^2_1 \circ \partial^3_2 = \gamma^2 \circ (\mathrm{Id} \wr \gamma^2)$.

Applying Proposition~\ref{prop:wreath_is_associative}, we may generalize $\tilde\Delta$
to the following maps.
\begin{definition}
Given a symmetric sequence with cocomposition $(A,\, \tilde\Delta)$ define 
associated natural transformations
$\tilde\Delta^n_i: A^{\bcirc (n-1)} \partial^n_i \to A^{\bcirc n}$, for $1\le i\le (n-1)$,
 which apply $\tilde\Delta$ 
at position $i$.  (Thus $\tilde\Delta = \tilde\Delta^2_1$.)
\end{definition}

These natural transformations induce coface maps in the following manner.
Since $\gamma^{n-1}\,\partial^n_i = \gamma^n$ and $\partial^n_i$ is epi,
transformations $B\,\gamma^{n} \to A^{\bcirc (n-1)}\,\partial^n_i$
are equivalent to transformations $B\gamma^{n-1} \to A^{\bcirc (n-1)}$ (where
$B:\Sigma_\ast \to \catC$ is some symmetric sequence).  Therefore
there is an equality of right Kan extensions
$\Rkan_{\gamma^n}\bigl(A^{\bcirc (n-1)}\,\partial^n_i\bigr) = 
 \Rkan_{\gamma^{n-1}}\bigl(A^{\bcirc(n-1)}\bigr) = A^{\pbcirc (n-1)}$.
We will make extensive use of this equality in later sections without further comment.

Define $\Delta^n_i:A^{\pbcirc (n-1)} \to A^{\pbcirc n}$ to be the following map.
\begin{equation}\begin{aligned}
\xymatrixnocompile@R=5pt@C=20pt{
     A^{\pbcirc (n-1)} \ar@{{}{=}{}}[d]  &&
 A^{\pbcirc n} \ar@{{}{=}{}}[d]   
\\ 
 \Rkan_{\gamma^n}\bigl(A^{\bcirc (n-1)}\,\partial^n_i\bigr) 
     \ar[rr]^(.55){\Rkan_{\gamma^n} (\tilde\Delta^n_i)} &&
 \Rkan_{\gamma^n}\bigl(A^{\bcirc n}\bigr)
} 
\end{aligned}\end{equation}

Under right Kan extension, 
Diagram~(\ref{eqn:cocomp}) translates to the following diagram of symmetric sequences.
\begin{equation}\begin{aligned}
\xymatrixnocompile@R=2pt{
& A\pbcirc A \ar@/^/[dr]^(.45){\Delta^3_1} & \\
A  
   \ar@/^/[ur]^(.4){\Delta} \ar@/_/[dr]_(.4){\Delta} & &
  \hskip -10pt  A\pbcirc A\pbcirc A \\
& A\pbcirc A  \ar@/_/[ur]_(.45){\Delta^3_2} & 
} 
\end{aligned}\end{equation}
Combined with 
Proposition~\ref{prop:wreath_is_associative}, this generalizes to the following.

\begin{proposition}\label{prop:delta_equalizer}
Let $(A,\, \tilde\Delta)$  be a symmetric sequence with cocomposition.
Then the transformation $\Delta^n_i:A^{\pbcirc (n-1)}\to A^{\pbcirc n}$ equalizes the two
 transformations 
$\Delta^{n+1}_i,\,\Delta^{n+1}_{i+1}:A^{\pbcirc n} \rightrightarrows A^{\pbcirc (n+1)}$.

More generally, 
$ \Delta^{n+1}_{j}\, \Delta^n_i = \Delta^{n+1}_i\, \Delta^n_{j-1}$ for $j > i$.
\end{proposition}
\qed

\begin{corollary}
Let $(A,\, \tilde\Delta)$  be a symmetric sequence with cocomposition.
There are canonical, unique maps
 $\Delta^{[n]} : A \to A^{\pbcirc n}$.
(Given by taking any chain of compositions $\Delta^n_{i_n}\cdots\Delta^1_{i_1}$.)
\end{corollary}

\subsection{Counit and Codegeneracies}

Write $\Id$ for the functor $\Id:\Sigma_\ast \to \catC$ given by
\[
\Id(T) = \begin{cases} 1_{\!\otimes\!} & \mathrm{if\ } |T| = 1, \\
                                    \star_\catC & \mathrm{otherwise}. \end{cases}
\]
We will call $\Id$ the ``counit'' symmetric sequence.  [The dual 
definition of the ``unit'' symmetric sequence would use $\emptyset_\catC$.]

\begin{definition}\label{D:unital symm seq}
A counital symmetric sequence is $(A,\tilde\epsilon)$ where $A$ is a symmetric sequence and 
$\tilde\epsilon$ is a natural transformation to the counit
$\tilde\epsilon:A\to \Id$. 
\end{definition}

Note that being counital is equivalent to the existence of a map 
$ A(1) \to 1_{\!\otimes\!}$.
We will not require the map $A(1) \to 1_{\!\otimes\!}$ to be equipped with a section.
In the next subsection, we will use the following basic equality whose proof can be read off of Figure~\ref{fig:degeneracy}.

\begin{lemma}\label{L:epsilon_degen}
The following functors $\Sigma_\ast \to \catC$ are equal.
\[
\bigl(\Id \bcirc A\bigr) s^1_0 \ = \ A \ = \ \bigl(A \bcirc \Id\bigr)s^1_1.
\]

More generally, the following functors $\Sigma_\ast^{\wr n} \to \catC$ are equal.
\[
\Bigl(\bigl(A^{\bcirc i}\bigr) \bcirc \Id \bcirc \bigl(A^{\bcirc (n-i)}\bigr)\Bigr)\,s^n_i 
  = A^{\bcirc n}
\]
\end{lemma}
\qed

In the footsteps of Lemma~\ref{L:epsilon_degen} we define the following generalization.

\begin{definition}
Given a counital symmetric sequence $(A,\, \tilde \epsilon)$ define associated 
natural transformations
$\tilde \epsilon^n_i:A^{\bcirc (n+1)} s^n_i \to A^{\bcirc n}$, for $0\le i\le n$, to be the following compositions.
\[
\xymatrixnocompile@R=5pt@C=55pt{
 A^{\bcirc (n+1)} s^n_i  \ar@{{}{=}{}}[d]   &
 A^{\bcirc n} \vphantom{s^n_i} \ar@{{}{=}{}}[d]   
\\ 
 \Bigl(\bigl(A^{\bcirc i}\bigr) \bcirc A \bcirc \bigl(A^{\bcirc (n-i)}\bigr)\Bigr)\, s^n_i 
     \ar[r]^{(\mathrm{Id} \bcirc \tilde\epsilon \bcirc \mathrm{Id})\,s^n_i } &
 \Bigl(\bigl(A^{\bcirc i}\bigr) \bcirc \Id \bcirc \bigl(A^{\bcirc (n-i)}\bigr)\Bigr)\,s^n_i 
} 
\]
Define $\tilde\epsilon^0_0 = \tilde\epsilon: A \to \Id$.
\end{definition}

These natural transformations induce codegeneracies in the following manner.
Since $\gamma^{n+1}\,s^n_i = \gamma^n$, the universal transformation
$A^{\pbcirc(n+1)}\,\gamma^{n+1} \to A^{\bcirc(n+1)}$ induces a transformation
$A^{\pbcirc(n+1)} \to \Rkan_{\gamma^n}\bigl(A^{\bcirc(n+1)}\,s^n_i\bigr)$.
Define $\epsilon^n_i:A^{\pbcirc (n+1)} \to A^{\pbcirc n}$ to be the following composition.
\begin{equation}\begin{aligned}
\xymatrixnocompile@R=5pt@C=20pt{
 &&&
 A^{\pbcirc n} \ar@{{}{=}{}}[d]   
\\
	A^{\pbcirc (n+1)} \ar[r] &
 \Rkan_{\gamma^n}\bigl(A^{\bcirc (n+1)}\, s^n_i\bigr) 
     \ar[rr]^(.55){\Rkan_{\gamma^n} (\tilde\epsilon^n_i)} &&
 \Rkan_{\gamma^n}\bigl(A^{\bcirc n}\bigr)
} 
\end{aligned}\end{equation}

Similar to Proposition~\ref{prop:delta_equalizer}, the corresponding properties of $s^n_i$
imply the following.

\begin{proposition}\label{prop:cosimplicial_identity_codegen}
Let $(A,\tilde\epsilon)$ be a counital symmetric sequence.  Then the transformation
$\epsilon^{n-1}_i:A^{\pbcirc n} \to A^{\pbcirc (n-1)}$ coequalizes the two transformations
$\epsilon^n_i,\,\epsilon^n_{i+1}:A^{\pbcirc (n+1)} \rightrightarrows A^{\pbcirc n}$.

More generally $\epsilon^{n-1}_i\,\epsilon^{n}_j = \epsilon^{n-1}_{j-1}\,\epsilon^{n}_i$ for $j > i$.
\end{proposition}
\qed

\subsection{Cooperads and Cosimplicial Structure}

\begin{definition}\label{def:cooperad}
A cocomposition operation on a counital symmetric sequence respects the counit if
the following diagram of natural transformations $\Sigma_\ast \to \catC$ commutes.
\begin{equation}\label{diag:coface}\begin{aligned}
\xymatrixnocompile@R=3pt@C=10pt{
& A \gamma^2 s^1_0  \ar[rr]^(.45){\tilde\Delta s^1_0} && (A\bcirc A) s^1_0 \ar[rrr]^(.5){(\tilde\epsilon\bcirc\mathrm{Id}) s^1_0} &&& (\Id\bcirc A) s^1_0 \ar@/^/[dr]^(.6){=} & \\
A \ar[rrrrrrr]|{\mathrm{\,Id\,}}  
   \ar@/^/[ur]^(.35){=} \ar@/_/[dr]_(.35){=} & & & & & & &
  A \\
& A \gamma^2 s^1_1  \ar[rr]_(.45){\tilde\Delta s^1_1} && (A\bcirc A) s^1_1 \ar[rrr]_(.5){(\mathrm{Id}\bcirc\tilde\epsilon) s^1_1} &&& (A\bcirc \Id) s^1_1 \ar@/_/[ur]_(.6){=} & \\
}
\end{aligned}\end{equation}

A counital cooperad is a counital symmetric sequence with 
cocomposition which respects the counit.
\end{definition} 

Applying Proposition~\ref{prop:wreath_is_associative} and using the simplicial structure of
wreath product categories, the requirement in Definition~\ref{def:cooperad} implies a more
general statement.

\begin{proposition}\label{prop:cosimplicial_identity_others}
If $(\cO,\, \tilde\Delta,\, \tilde \epsilon)$ is a cooperad, then
the following composition is equal to the identity 
$\mathrm{Id}_{\cO^{\bcirc n}}$, for $j=(i-1),\, i$.
\[ 
\cO^{\bcirc n} = \cO^{\bcirc n}\, \partial^{n+1}_i s^n_j 
   \xrightarrow{\ \tilde \Delta^{n+1}_i\, s^n_j\ }
\cO^{\bcirc (n+1)}\, s^n_j
   \xrightarrow{\ \tilde \epsilon^n_j\ }
\cO^{\bcirc n}
\]
Furthermore, the following compositions are equal if $j < i-1$.
\[
\xymatrixnocompile@R=5pt@C=40pt{
\cO^{\bcirc n}\, (\partial^{n+1}_i\,s^n_j) 
   \ar[r]^(.51){\tilde\Delta^{n+1}_i\,s^n_j}   
   \ar@{{}{=}{}}[d] &
\cO^{\bcirc (n+1)}\, s^n_j 
   \ar[r]^(.6){\tilde\epsilon^n_j} &
\cO^{\bcirc n} \\
\cO^{\bcirc n}\, (s^{n-1}_{j} \partial^n_{i+1})
   \ar[r]_(.55){\tilde \epsilon^{n-1}_j\, \partial^n_{i+1}} &
\cO^{\bcirc (n-1)} \partial^n_{i+1}
   \ar[r]_(.6){\tilde \Delta^n_{i+1}} &
\cO^{\bcirc n}
}
\]
as well as the similar statement for $j > i$.
\end{proposition}
\qed

We have now almost completed the proof of the following.

\begin{theorem}\label{thm:cO_complex}
If $(\cO,\, \tilde\Delta,\, \tilde\epsilon)$ is a cooperad, then the collection
$\{\cO^{\pbcirc n}\}_n$ along with coface maps $\Delta^n_i$ and codegeneracy maps
$\epsilon^n_i$ defines a coaugmented cosimplicial symmetric sequence with two extra
 codegeneracies.
\begin{equation}\label{E:O complex}\begin{aligned}\xymatrix@C=2pt{
 \cO \ar[rrrr]
    &&\quad&\ar@{-->}@<3pt>[ll] \ar@{-->}@<-3pt>[ll] &
 \cO\pbcirc \cO \ar@<3pt>[rrrr] \ar@<-3pt>[rrrr]
    &&\quad&\ar@{-->}[ll] \ar@{-->}@<6pt>[ll] \ar@{-->}@<-6pt>[ll] &
 \cO^{\pbcirc 3} \ar@<6pt>[rrrr] \ar[rrrr] \ar@<-6pt>[rrrr]
    &&\quad&\ar@{-->}@<3pt>[ll] \ar@{-->}@<-3pt>[ll] \ar@{-->}@<9pt>[ll] \ar@{-->}@<-9pt>[ll] &
 \cO^{\pbcirc 4} \ar@<9pt>[rrrr] \ar@<3pt>[rrrr] \ar@<-3pt>[rrrr] \ar@<-9pt>[rrrr]
    &&\quad&\ar@{-->}[ll] \ar@{-->}@<6pt>[ll] \ar@{-->}@<-6pt>[ll] 
            \ar@{-->}@<12pt>[ll] \ar@{-->}@<-12pt>[ll] & 
 \cdots  
}\end{aligned}\end{equation}
\end{theorem}
\begin{proof}
In Propositions~\ref{prop:delta_equalizer} and 
\ref{prop:cosimplicial_identity_codegen}, we have already shown the cosimplicial identities
$ \Delta^{n+1}_{j}\, \Delta^n_i = \Delta^{n+1}_i\, \Delta^n_{j-1}$ 
and $\epsilon^{n-1}_i\,\epsilon^{n}_j = \epsilon^{n-1}_{j-1}\,\epsilon^{n}_i$.

It remains only to consider the compositions $\Delta^{n+1}_i\, \epsilon^n_j$.
These come from the right Kan extension over $\gamma^n$ of the statements of 
Proposition~\ref{prop:cosimplicial_identity_others}.  Note that the right Kan extension
$\Rkan_{\gamma^n}\Bigl(\cO^{\bcirc n}\, \partial^{n+1}_i s^n_j 
 \xrightarrow{\ \tilde\Delta^{n+1}_i \, s^n_j\ } \cO^{\bcirc (n+1)}\, s^n_j\Bigr)$ 
is equal to the composition
\[
 \Rkan_{\gamma^{n+1}}\bigl(\cO^{\bcirc n}\, \partial^{n+1}_i\bigr) 
  \xrightarrow{\ \Delta^{n+1}_i\ } 
 \Rkan_{\gamma^{n+1}}\bigl(\cO^{\bcirc (n+1)}\bigr)
  \longrightarrow
 \Rkan_{\gamma^n}\bigl(\cO^{\bcirc (n+1)}\, s^n_j\bigr).
\]
\end{proof}

\subsection{Parenthesization Maps and Cooperad Structure.}

From now on, let $A, B, C$ be generic symmetric sequences and
$(\cO,\,\tilde\Delta,\,\tilde\epsilon)$ be a generic counital cooperad.

\begin{proposition}\label{L:parenthesizations}
There are canonical 
``parenthesization'' natural transformations:
$$\begin{aligned}\xymatrix@C=20pt@R=1pt{
  (A\pbcirc B)\pbcirc C \ar@/^/[dr] & \\
  & A\pbcirc B\pbcirc C  \\ 
  A\pbcirc( B\pbcirc C) \ar@/_/[ur] & 
}\end{aligned}$$
More generally there are parenthesization maps 
to $A_1\pbcirc\cdots\pbcirc A_n$
from any parenthesization of this expression.
\end{proposition}

\begin{proof}
We show the existence of the map
$(A\pbcirc B) \pbcirc C \to A\pbcirc B\pbcirc C$.  The other maps are similar.

The universal natural transformation $(A \pbcirc B)\,\gamma^2 \longrightarrow (A\bcirc B)$ induces
a natural transformation of functors 
 $(\Sigma_\ast \wr \Sigma_\ast) \wr \Sigma_\ast \longrightarrow \catC$:
$$ \bigl((A\pbcirc B) \bcirc C\bigr) \partial^3_1
 \longrightarrow (A\bcirc B)\bcirc C
= A\bcirc B\bcirc C
.$$
The desired map is induced by taking the right Kan extension $\Rkan_{\gamma^3}$ 
of the diagram above.
\[
\xymatrixnocompile@R=5pt@C=55pt{
 (A\pbcirc B)\pbcirc C  \ar@{{}{=}{}}[d]   &
 A \pbcirc B \pbcirc C \ar@{{}{=}{}}[d]   
\\ 
 \Rkan_{\gamma^3} \Bigl(\bigl( (A\pbcirc B)\bcirc C\bigr)\,\partial^3_1\Bigr)
     \ar[r] &
 \Rkan_{\gamma^3} (A\bcirc B\bcirc C)
} 
\]
\end{proof}

\begin{remark}[On the associativity of $\pbcirc$]
Without making further assumptions, it is not true that 
$(A\pbcirc B) \pbcirc C \cong A\pbcirc B\pbcirc C \cong A\pbcirc(B\pbcirc C)$.
This would follow from the existence of natural equivalences
$\bigl(\Rkan_{\gamma^2} (A\bcirc B)\bigr) \bcirc C\cong 
\Rkan_{\partial^3_1} (A\bcirc B \bcirc C)$ 
as well as the corresponding equivalence using 
$\partial^3_2$.  However, this will generally only
occur if the symmetric monoidal product $\otimes$ of $\catC$ commutes with products.

The situation contrasts starkly with that of the operad composition product, 
defined dual to $\pbcirc$ using left rather than 
right Kan extensions.  If $\catC$ is a closed monoidal category, then $\otimes$ 
is a left adjoint, so it will in particular commute with coproducts and left Kan
extensions.  In this case the parenthesization maps for the operad composition product
are  isomorphisms and the operad composition product is associative.
\end{remark}

\begin{proposition}\label{prop:parenthesization_is_associative}
Parenthesization maps are associative.
\end{proposition}

For example the following diagrams commute.
\begin{equation}\label{diag:associative1}\begin{aligned}
\xymatrix@C=10pt@R=1pt{
& (A\pbcirc B\pbcirc C)\pbcirc D \ar@/^1pc/[dr] & \\
\bigl((A\pbcirc B)\pbcirc C\bigr)\pbcirc D  \ar@/^.7pc/[ur] \ar@/_.7pc/[dr] && 
  A\pbcirc B\pbcirc C\pbcirc D \\
& (A\pbcirc B)\pbcirc C\pbcirc D \ar@/_1pc/[ur] & 
}
\end{aligned}\end{equation}
\begin{equation}\label{diag:associative2}\begin{aligned}
\xymatrix@C=10pt@R=1pt{
& (A\pbcirc B)\pbcirc C\pbcirc D \ar@/^1pc/[dr] & \\
(A\pbcirc B)\pbcirc (C\pbcirc D)  \ar@/^.7pc/[ur] \ar@/_.7pc/[dr] && 
  A\pbcirc B\pbcirc C\pbcirc D \\
& A\pbcirc B\pbcirc (C\pbcirc D) \ar@/_1pc/[ur] & 
}
\end{aligned}\end{equation}

\begin{proof}[Proof of \ref{prop:parenthesization_is_associative}]
It is enough to consider Diagrams~(\ref{diag:associative1}) and (\ref{diag:associative2}).  
Commutativity is shown by writing the diagrams as right Kan extensions.  The diagrams
above are $\Rkan_{\gamma^4}$ of the following diagrams of functors 
$\Sigma^{\wr 4}_\ast \to \catC$.
\begin{equation*}\begin{aligned}
\xymatrix@C=0pt@R=1pt{
& \bigl((A\pbcirc B\pbcirc C)\bcirc D\bigr)\,(\gamma^3\wr\mathrm{Id})  
      \ar@/^1.5pc/[dr] & \\
\Bigl(\bigl((A\pbcirc B)\pbcirc C\bigr)\bcirc D\Bigr)\,(\gamma^3\wr\mathrm{Id}) \hskip -20pt
       \ar@/^1pc/[ur] \ar@/_1pc/[dr]  &&  \hskip -20pt
  A\bcirc B\bcirc C\bcirc D \\
& \bigl((A\pbcirc B)\bcirc C\bcirc D\bigr)\,\partial^4_1 \ar@/_1pc/[ur] & 
}
\end{aligned}\tag{\ref{diag:associative1}'}\end{equation*}
\begin{equation*}\begin{aligned}
\xymatrix@C=0pt@R=1pt{
& \bigl((A\pbcirc B)\bcirc C\bcirc D\bigr)\,\partial^4_1  
      \ar@/^1pc/[dr] & \\
\bigl((A\pbcirc B)\bcirc (C\pbcirc D)\bigr)\,(\gamma^2\wr\mathrm{Id}\wr\gamma^2) 
       \hskip -20pt
       \ar@/^1pc/[ur] \ar@/_1pc/[dr]  &&  \hskip -10pt
  A\bcirc B\bcirc C\bcirc D \\
& \bigl(A\bcirc B\bcirc (C\pbcirc D)\bigr)\,\partial^4_3 \ar@/_1pc/[ur] & 
}
\end{aligned}\tag{\ref{diag:associative2}'}\end{equation*}
Diagram~(\ref{diag:associative1}') is just $-\bcirc D$ applied to the following universal diagram
(in which the upper-left map is $\Rkan_{\gamma^3}$ of the lower-right).
\begin{equation*}\begin{aligned}
\xymatrix@C=10pt@R=1pt{
& (A\pbcirc B\pbcirc C)\,\gamma^3  
      \ar@/^1pc/[dr] & \\
\bigl((A\pbcirc B)\pbcirc C\bigr)\,\gamma^3 
       \ar@/^.7pc/[ur] \ar@/_.7pc/[dr]  &&  
  A\bcirc B\bcirc C \\
& \bigl((A\pbcirc B)\bcirc C\bigr)\,\partial^3_1 \ar@/_1pc/[ur] & 
}
\end{aligned}\tag{\ref{diag:associative1}''}\end{equation*}
Diagram~(\ref{diag:associative2}') commutes because the upper and lower composition
are both equal to
\[\bigl((A\pbcirc B)\bcirc (C\pbcirc D)\bigr)\,(\gamma^2\wr\mathrm{Id}\wr\gamma^2)
  \xrightarrow{\ \ \iota_1 \bcirc \iota_2\ \ }
 (A\bcirc B)\bcirc (C\bcirc D)
\]
Where $\iota_1:(A\pbcirc B)\,\gamma^2 \to A\bcirc B$ and 
$\iota_2:(C\pbcirc D)\,\gamma^2 \to C\bcirc D$ are the universal natural transformations
from their respective Kan extensions.
\end{proof}

We relate parenthesization maps with cooperad structure.
By the functoriality of $\bcirc$, there are natural transformations 
$\mathrm{Id} \bcirc \Delta: A \bcirc \cO \longrightarrow A \bcirc (\cO \pbcirc \cO)$ and
$\Delta \bcirc \mathrm{Id}: \cO \bcirc A \longrightarrow (\cO \pbcirc \cO) \bcirc A$, 
where $A$ is any
symmetric sequence.  Define the maps
$\mathrm{Id} \pbcirc \Delta$ and $\Delta \pbcirc \mathrm{Id}$ to be the natural
transformations
induced on right Kan extensions via functoriality of Kan extension.  For example
\[\Delta\pbcirc\mathrm{Id}  = \Rkan_{\gamma^2} \bigl(\Delta\bcirc \mathrm{Id} \bigr):
\cO\pbcirc A   \longrightarrow  (\cO \pbcirc \cO)\pbcirc A .\] 
By alternately letting $A$ be a parenthesization of $\cO^{\pbcirc k}$ 
and using functoriality of $\pbcirc$
this defines maps from any parenthesization of $\cO^{\pbcirc n}$.  For example
\[\bigl((\mathrm{Id}\pbcirc\mathrm{Id})\pbcirc \Delta\bigr)\pbcirc\mathrm{Id}:
\bigl((\cO\pbcirc\cO)\pbcirc\cO\bigr)\pbcirc\cO \longrightarrow
\bigl((\cO\pbcirc\cO)\pbcirc(\cO\pbcirc\cO)\bigr)\pbcirc\cO.\]

\begin{theorem}
The following diagrams commute  (unlabeled maps are parenthesization).
\[
\begin{aligned}\xymatrix@C=0pt@R=10pt{
 &\ \ & (\cO\pbcirc\cO)\pbcirc\cO \ar@/^1pc/[dr] & \\
 \cO\pbcirc\cO \ar@/^1pc/[urr]^(.32){\Delta\pbcirc\mathrm{Id}} \ar[rrr]_{\Delta^3_1} &&&
\cO\pbcirc\cO\pbcirc\cO
}\end{aligned} \qquad
\begin{aligned}\xymatrix@C=0pt@R=10pt{
 &\ \ & \cO\pbcirc(\cO\pbcirc\cO) \ar@/^1pc/[dr] & \\
\cO\pbcirc\cO \ar@/^1pc/[urr]^(.32){\mathrm{Id}\pbcirc\Delta} \ar[rrr]_{\Delta^3_2} &&&
\cO\pbcirc\cO\pbcirc\cO
}\end{aligned}
\]
More generally, parenthesization maps convert $\mathrm{Id}\pbcirc\Delta\pbcirc\mathrm{Id}$
(and its parenthesizations) to $\Delta^4_2$, etc.
\end{theorem}
\begin{proof}
We show the first diagram commutes.  The second diagram and more general
statement are proven the same.

Consider the diagram below, where maps marked $\iota$ are all universal 
transformations of right Kan extensions $(\Rkan_F X)\,F \xrightarrow{\ \iota\ } X$.
\begin{equation}\label{diag:big}\begin{aligned}\xymatrixnocompile@C=2pt@R=10pt{
 & \bigl((\cO\pbcirc \cO)\pbcirc \cO\bigr)\,\gamma^3 
   \ar[rr]^{\iota\,\partial^3_1}
   \ar@{{}{ }{}}[dr]|{\text{\textcircled{1}}}  
 & \ar@{{}{ }{}}[ddrr]^{\text{\textcircled{2}}} 
 & \bigl((\cO\pbcirc \cO)\bcirc \cO\bigr)\, \partial^3_1 \hskip -30pt 
   \ar[ddr]^{\iota\bcirc \mathrm{Id}} & 
\\   
 \bigl(\cO \pbcirc \cO\bigr)\,\gamma^3 
    \ar[ur]^(.35){(\Delta\pbcirc\mathrm{Id})\,\gamma^3} 
    \ar[rr]^(.4){\iota\,\partial^3_1}    
    \ar[drr]_{\Delta^3_1\,\gamma^3}
 & & (\cO\bcirc\cO)\,\partial^3_1 
    \ar[drr]^{\tilde\Delta\bcirc\mathrm{Id}} 
    \ar[ur] ^(.35){(\Delta\bcirc\mathrm{Id})\,\partial^3_1}
    \ar@{{}{ }{}}[d]|{\text{\textcircled{3}}}
 & & 
\\  
 & & (\cO\pbcirc\cO\pbcirc\cO)\, \gamma^3 \ar[rr]_{\iota}
 & &\cO \bcirc \cO \bcirc \cO  
}\end{aligned}\end{equation}
Parallelograms \textcircled{1} and \textcircled{3} commute by functoriality of right Kan
extension.  The left side of parallelogram \textcircled{1} is 
$\Rkan_{\gamma^2}$ of the right side,
and the left side of parallelogram \textcircled{3} is $\Rkan_{\gamma^3}$ of the right side.
Triangle \textcircled{2} commutes by functoriality of $\bcirc$ 
(recall that $\iota\Delta = \tilde\Delta$). 

Applying $\Rkan_{\gamma^3}$ along the outside of Diagram~(\ref{diag:big}) yields the
 following (where the map labeled $\ast$ is the parenthesization map).
\begin{equation}\begin{aligned}\xymatrixnocompile@C=2pt@R=10pt{
 & (\cO\pbcirc \cO)\pbcirc \cO
   \ar[rr]^{=}  
 &  & (\cO\pbcirc \cO)\pbcirc \cO \hskip -10pt 
   \ar[ddr]^{\ast} & 
\\   
 \cO \pbcirc \cO
    \ar[ur]^(.35){\Delta\pbcirc\mathrm{Id}}   
    \ar[drr]_{\Delta^3_1}
 & & 
 & & 
\\  
 & & \cO\pbcirc\cO\pbcirc\cO \ar[rr]_{=}
 & &\cO \pbcirc \cO \pbcirc \cO  
}\end{aligned}\end{equation}
\end{proof}

\begin{example}
The following diagram is commutative (the unlabeled maps are parenthesizations).
 $$\xymatrixnocompile@R=14pt{
  (\cO \pbcirc \cO) \pbcirc \cO 
      \ar[rr]^(.45){\left(\Delta\pbcirc\mathrm{Id}\right)\pbcirc\mathrm{Id}} 
      \ar[d] & \qquad &
    \bigl((\cO \pbcirc \cO)\pbcirc \cO\bigr)\pbcirc \cO \ar[d] \ar[dr] & \\
  \cO\pbcirc\cO\pbcirc\cO \ar[rr]^(.45){\Delta\pbcirc\mathrm{Id}\pbcirc\mathrm{Id}} 
      \ar@/_15pt/[rrr]_{\Delta^4_1} & &
    (\cO \pbcirc \cO) \pbcirc \cO \pbcirc \cO \ar[r] &
  \cO\pbcirc\cO\pbcirc\cO\pbcirc\cO
}$$
\end{example}

\section{Comodules and Coalgebras}

Throughout this section, 
let $(\cO, \tilde \Delta_\cO, \tilde\epsilon)$ be a counital cooperad and 
$M$ be a symmetric sequence.

\subsection{Comodules}

\begin{definition}
 A left $\cO$-comodule is $(M,\,\tilde\Delta_M)$ where $M$ is a symmetric sequence and
$\tilde \Delta_M:M\,\gamma_2 \to \cO\bcirc M$ is compatible with 
$\partial^3_1$ and $\partial^3_2$ and $s^1_0$.
\end{definition}

That is, 
 the following diagrams (analogous to Diagrams~(\ref{eqn:cocomp}) and (\ref{diag:coface}))
 should commute.
\begin{equation} \label{eqn:comod-coface}\begin{aligned}
\xymatrixnocompile@R=2pt{
& (\cO\bcirc M)\, (\gamma^2\wr \text{Id}) 
   \ar@/^/[dr]^(.6){\tilde\Delta_\cO\wr\text{Id}} & \\
M \,\gamma^3 
   \ar@/^/[ur]^(.35){\tilde\Delta_M} \ar@/_/[dr]_(.35){\tilde\Delta_M} & &
  \hskip -10pt \cO\bcirc \cO\bcirc  M\\
& (\cO\bcirc M)\, (\text{Id}\wr \gamma^2)
   \ar@/_/[ur]_(.6){\text{Id}\wr\tilde\Delta_{M}} & 
}\end{aligned}\end{equation}
\begin{equation} \label{eqn:comod-codegen}\begin{aligned}
\xymatrixnocompile@R=3pt@C=10pt{
& M\, \gamma^2 s^1_0  
    \ar[rr]^(.45){\tilde\Delta_M\, s^1_0} 
&& (\cO\bcirc M)\, s^1_0 
    \ar[rrr]^(.5){(\tilde\epsilon\bcirc \mathrm{Id})\, s^1_0} 
&&& (\Id\bcirc M)\, s^1_0 
    \ar@/^/[dr]^(.7){=} & \\
M \ar[rrrrrrr]|{\mathrm{\,Id\,}}  
    \ar@/^/[ur]^(.30){=} 
 & & & & & & & M \\
}
\end{aligned}\end{equation}

As with cooperads, we write $\Delta_M$ 
for the induced universal transformation to the right Kan extension 
$\Delta_M: M\to \cO\pbcirc M$. 
There are induced transformations 
$\tilde\Delta^{n+1}_{i}:\bigl(\cO^{\bcirc (n-1)}\bcirc M\bigr)\,\partial^{n+1}_i\to \cO^{\bcirc n}\bcirc M$ and
$\Delta^{n+1}_{i}:\cO^{\pbcirc (n-1)}\pbcirc M\to \cO^{\pbcirc n}\pbcirc M$.

\begin{theorem}\label{L:module complex}
Analogous to Theorem~\ref{thm:cO_complex} there is a canonical 
coaugmented cosimplicial complex as below.
$$\xymatrix@C=2pt{
 M \ar[rrrr]
    &&\quad&\ar@{-->}@<3pt>[ll]  &
 \cO\pbcirc M \ar@<3pt>[rrrr] \ar@<-3pt>[rrrr]
    &&\quad&\ar@{-->}[ll] \ar@{-->}@<6pt>[ll] &
 \cO^{\pbcirc 2}\pbcirc M \ar@<6pt>[rrrr] \ar[rrrr] \ar@<-6pt>[rrrr]
    &&\quad&\ar@{-->}@<3pt>[ll] \ar@{-->}@<-3pt>[ll] \ar@{-->}@<9pt>[ll] &
 \cO^{\pbcirc 3}\pbcirc M \ar@<9pt>[rrrr] \ar@<3pt>[rrrr] \ar@<-3pt>[rrrr] \ar@<-9pt>[rrrr]
    &&\quad&\ar@{-->}[ll] \ar@{-->}@<6pt>[ll] \ar@{-->}@<-6pt>[ll] 
            \ar@{-->}@<12pt>[ll] & 
 \cdots  
}$$
\end{theorem}
\qed

\begin{corollary}\label{C:mu^M_n exists}
 There are unique 
transformations
$\Delta^{[n]}_{M}:M\to \cO^{\pbcirc (n-1)}\pbcirc M$. These are equal to any combination
of parenthesization maps and cocomposition maps from their source to their target.
\end{corollary}
\qed

\subsection{Coalgebras}

Let $a$ be an object of $\catC$ and $A$ be a symmetric sequence. Note that 
$a$ can be viewed as a functor $a:\Sigma_0 \to \catC$.  Recall the descriptions of
the category 
$\hat \emptyset_n$ in Remarks~\ref{rem:hat_empty2} and \ref{rem:hat_empty1}.
We may view $\hat \emptyset_n$ either as the category of level $n$ trees with no 
leaves; or as $\overline{\Sigma}^{\wr (n-1)}_\ast \subset \Sigma^{\wr (n+1)}_\ast$, the
full subcategory consisting
of chains of set maps of the following form.
\[
\star \xleftarrow{f_0} 
  S_1 \xleftarrow{f_1} \cdots \xleftarrow{f_{n-2}} S_{n-1} 
\xleftarrow{f_n} \emptyset
\]
Note that the category $\overline{\Sigma}^{\wr 0}_\ast$ consists of only the trivial
 chain $(\star \leftarrow \emptyset)$.  This is equivalent to $\Sigma_0$.

The face and degeneracy maps of $\Sigma^{\wr (n+1)}_\ast$ induce 
the following face and 
degeneracy maps on $\overline{\Sigma}^{\wr (n-1)}_\ast$.
(We introduce an index shift below so that $\bar \partial^n_i$ and $\bar s^n_j$ map from
$\hat\emptyset_n = \overline{\Sigma}^{\wr (n-1)}_\ast$.)
\[\begin{cases}
 \bar\partial^n_i:\overline{\Sigma}^{\wr (n-1)}_\ast \to \overline{\Sigma}^{\wr (n-2)}_\ast,
     &\mathrm{for}\ 1\le i\le (n-1),\ \mathrm{and}\ n> 1 \\
 \bar s^n_i:\overline{\Sigma}^{\wr (n-1)}_\ast \to \overline{\Sigma}^{\wr n}_\ast,
     &\mathrm{for}\ 0\le i\le n\ \mathrm{and}\ n \ge 1
\end{cases}\] 
The degeneracy map $\bar s^n_n$ doubles $\emptyset$, recognizing that a tree without
leaves of level $n$ is also of level $(n+1)$.  
Note that $\bar\partial^2_1:\overline{\Sigma}^{\wr 1}_\ast \to \Sigma_0$
coequalizes all chains of face maps from $\overline{\Sigma}^{\wr (n-1)}_\ast$ to 
$\overline{\Sigma}^{\wr 1}_\ast$.  We write $\bar \gamma^n$ for the composition
$\bar\gamma^n = (\bar\partial^2_{i_2}\cdots\bar\partial^n_{i_n})$.

Under the identification $\overline{\Sigma}^{\wr n}_\ast \subset \Sigma^{\wr (n+2)}_\ast$,
Definition~\ref{D:composition product} of symmetric sequence composition restricts to a functor 
$(A_1\bcirc \cdots A_{n-1}\bcirc a):\overline{\Sigma}^{\wr (n-1)}_\ast \to \catC$.
For example, $A\bcirc a$ is given by the following.
\begin{align*}
\bigl(\star \xleftarrow{f_0} S \xleftarrow{f_1} \emptyset \bigr) \longmapsto
  \ \ &A(S)\otimes\left(\bigotimes_{s\in S} a(\emptyset)\right) \\ 
&= A(S)\otimes a^{\otimes |S|}
\end{align*}
The right Kan extension of Definition~\ref{D:composition product} restricts to a 
right Kan extension over $\bar\gamma^n:\overline{\Sigma}^{\wr (n-1)}_\ast \to \Sigma_0$,
 yielding the following functor.
\[
A_1\pbcirc\cdots\pbcirc A_{n-1}\pbcirc a = 
 \Rkan_{\bar \gamma^n} (A_1\bcirc \cdots\bcirc A_{n-1}\bcirc a) :
 \Sigma_0 \longrightarrow \catC
\]
For example, 
$(A\pbcirc a) = 
\displaystyle \prod_{k\le 0} \Bigl(A(k) \otimes a^{\otimes k}\Bigr)^{\Sigma_k}.$

\begin{definition}
A coalgebra over the cooperad $(\cO, \tilde\Delta, \tilde \epsilon)$ is 
$(c,\tilde\Delta_c)$ where $c$ is an object of $\catC$ and 
$\tilde\Delta_c: c\, \bar\gamma^2 \to \cO \bcirc c$ is compatible with 
face maps 
$\bar\partial^2_1=(\gamma^2\wr \mathrm{Id})$, 
$\bar\partial^2_2=(\mathrm{Id}\wr \bar\gamma^2)$ and degeneracy
$\bar s^1_0$.
\end{definition}

That is, 
 the following diagrams (analogous to Diagrams~(\ref{eqn:comod-coface}) and
 (\ref{eqn:comod-codegen}))
 should commute.
\begin{equation} \label{eqn:coalg-coface}\begin{aligned}
\xymatrixnocompile@R=2pt{
& (\cO\bcirc c) (\gamma^2\wr \text{Id}) \ar@/^/[dr]^(.6){\tilde\Delta_\cO\wr\text{Id}} & \\
c\, \bar\gamma^3 
   \ar@/^/[ur]^(.35){\tilde\Delta_c} \ar@/_/[dr]_(.35){\tilde\Delta_c} & &
  \hskip -10pt \cO\bcirc \cO\bcirc  c\\
& (\cO\bcirc c) (\text{Id}\wr \bar\gamma^2) \ar@/_/[ur]_(.6){\text{Id}\wr\tilde\Delta_{c}} & 
}\end{aligned}\end{equation}
\begin{equation} \label{eqn:coalg-codegen}\begin{aligned}
\xymatrixnocompile@R=3pt@C=10pt{
& c\, \bar\gamma^2 \bar s^1_0  
    \ar[rr]^(.45){\tilde\Delta_c\, \bar s^1_0} 
&& (\cO\bcirc c)\, \bar s^1_0 
    \ar[rrr]^(.5){(\tilde\epsilon\bcirc \mathrm{Id})\, \bar s^1_0} 
&&& (\Id\bcirc c)\, \bar s^1_0 
    \ar@/^/[dr]^(.67){=} & \\
c \ar[rrrrrrr]|{\mathrm{\,Id\,}}  
    \ar@/^/[ur]^(.33){=} 
 & & & & & & & c \\
}
\end{aligned}\end{equation}

Statements and proofs about comodules translate into statements and proofs 
about coalgebras by converting $\partial^n_i$, $s^n_i$ into $\bar\partial^n_i$, 
$\bar s^n_i$.  Essentially, coalgebras are comodules which are concentrated in 
0-arity.   Write
$\Delta_c$ for the induced map (in $\catC$) $\Delta_c:c\to \cO\pbcirc c$.  As with comodules
we have
$\tilde\Delta^{n+1}_i:\bigl(\cO^{\bcirc (n-1)}\bcirc c\bigr)\,\bar\partial^{n+1}_i \to
 \cO^{\bcirc n}\bcirc c$ inducing
$\Delta^{n+1}_i:\cO^{\pbcirc(n-1)}\pbcirc c\to \cO^{\pbcirc n}\pbcirc c$.

\begin{theorem}\label{L:algebra complex}
The comultiplication $\Delta_c$ defines
a canonical 
coaugmented cosimplicial complex (in $\catC$)
$$\xymatrix@C=2pt{
 c \ar[rrrr]
    &&\quad&\ar@{-->}@<3pt>[ll]  &
 \cO\pbcirc c \ar@<3pt>[rrrr] \ar@<-3pt>[rrrr]
    &&\quad&\ar@{-->}[ll] \ar@{-->}@<6pt>[ll] &
 \cO^{\pbcirc 2}\pbcirc c \ar@<6pt>[rrrr] \ar[rrrr] \ar@<-6pt>[rrrr]
    &&\quad&\ar@{-->}@<3pt>[ll] \ar@{-->}@<-3pt>[ll] \ar@{-->}@<9pt>[ll] &
 \cO^{\pbcirc 3}\pbcirc c \ar@<9pt>[rrrr] \ar@<3pt>[rrrr] \ar@<-3pt>[rrrr] \ar@<-9pt>[rrrr]
    &&\quad&\ar@{-->}[ll] \ar@{-->}@<6pt>[ll] \ar@{-->}@<-6pt>[ll] 
            \ar@{-->}@<12pt>[ll] & 
 \cdots  
}$$
\end{theorem}

\begin{corollary}\label{C:mu^a_n exists}
There are unique 
$\catC$-maps
$\Delta^{[n]}:c \longrightarrow \cO^{\pbcirc (n-1)} \pbcirc c$.  
These are equal to any combination
of parenthesization maps and cocomposition maps from their source to their target.  
\end{corollary}

\section{Examples}

We end with a two simple examples of cooperads which are not duals of standard operads.
Both of these are constructed via quotient/contraction operations.  The (directed) graph 
cooperad is used in \cite{SiWa1} and the contractible $\Delta$ complex operad is a 
generalization.

\subsection{The Graph Cooperad}

Given a finite set $S$, a contractible $S$-graph is a connected, acyclic graph whose
vertex set is $S$.  The unoriented graph cooperad has $\overline{\textsc{gr}}(S)$ equal to the 
free $\mathbb{Z}$ module generated by all contractible $S$-graphs. 
The cocomposition natural transformation
$\tilde\Delta:\overline{\textsc{gr}}\,\gamma^2 \to
    \overline{\textsc{gr}}\bcirc\overline{\textsc{gr}}$
is defined as follows.

Given two graphs $G$ and $K$, a quotient map of graphs $q:G \twoheadrightarrow K$
is a surjective map from vertices of $G$ onto vertices of $K$ such that 
$q(v_1,v_2) = \bigl(q(v_1),q(v_2)\bigr)$ defines a map sending edges of $G$ to 
edges and vertices (if $q(v_1)=q(v_2)$) of $K$, surjecting onto the edges.  
Note that if $q:G\twoheadrightarrow K$ is a quotient map and $v$ is a vertex of $K$,
then $q^{-1}(v)$ is a subgraph of $G$.
A graph contraction is a quotient map where each $q^{-1}(v)$ is a connected subgraph. 
Note that there is a bijection between the edges of $G$ and the edges of $K$ union those of
the $q^{-1}(v)$.

Suppose $G$ is an $S$-graph and $f:S\twoheadrightarrow T$ is a surjection of sets.
Given $t\in T$, let $\overline{f^{-1}(t)}$ be the maximal subgraph of $G$ 
supported by the vertices of $f^{-1}(t)$.
We say that $f$ induces a graph contraction on $G$ if 
$\overline{f^{-1}(t)}$ is contractible for each $t$.
In this case, we define the induced contracted graph $(G/f)$ to have vertices $T$ with 
an edge from vertex $t_1$ to $t_2$ if there is an edge in $G$ from the subgraph
$\overline{f^{-1}(t_1)}$ to the subgraph $\overline{f^{-1}(t_2)}$.  

Cocomposition $\tilde\Delta$ takes the element 
$\bigl(T \xleftarrow{f} S \bigl)$ of
 $\Sigma_\ast\wr\Sigma_\ast$ to the map 
\[
\overline{\textsc{gr}}(S) \longrightarrow 
  \overline{\textsc{gr}}(T) \otimes
    \Bigl(\bigotimes_{t\in T} \overline{\textsc{gr}}(f^{-1}(t))\Bigr)
\]
which takes a $S$-graph $G$ to 
$(G/f)\otimes \bigl(\bigotimes_{t\in T} \overline{f^{-1}(t)}\bigr)$
if $f$ defines a graph contraction on $G$, and sends $G$ to 0 otherwise.
Since the quotient operation described previously is clearly associative, this defines a 
symmetric sequence with cocomposition. 
The counit map sends $S$-graphs with only one vertex to $1\in \mathbb{Z}$ and 
kills all others.

The (directed) graph cooperad is similar to the unoriented graph cooperad.  In the 
category of \underline{directed}, contractible $S$-graphs define
$\overrightharpoonup{\textsc{gr}}(S) = \overline{\textsc{gr}}(S)/\sim$, 
where $\sim$ identifies 
reversing the orientation of an edge with multiplication of a graph by $-1$.
Cocomposition on $\overline{\textsc{gr}}$ gives a well-defined map on 
$\overrightharpoonup{\textsc{gr}}$ since reversing an arrow in $G$ will 
reverse exactly one arrow either
in the quotient graph $G/f$ or in one of the $f^{-1}(t)$.

The graph cooperad generalizes to the following.

\subsection{The CDC Cooperad}

By a $\Delta$-complex, we mean what Hatcher \cite[Appendix]{Hat} calls a ``singular
$\Delta$-complex'' or $s\Delta$-complex''.  Essentially this is a CW complex whose cells
are all (oriented) simplices and whose attaching maps factor through face maps of the simplex.  
Given a set
$S$, an $S\,\Delta$-complex is a $\Delta$-complex whose 0-cells are labeled by elements of 
$S$.  The CDC cooperad has $\textsc{cdc}(S)$ equal to the free $\mathbb{Z}$ module 
generated by contractible $S\,\Delta$-complexes.  Cocomposition is defined similar to that
for $\overline{\textsc{gr}}$.  

If $T$ is a subset of the 0-cells of a $\Delta$-complex $X$, write $\overline{T}$ for the
maximal CW subcomplex of $X$ supported by $T$.
Quotient maps for $\Delta$-complexes are CW quotient maps.  We say a quotient map
$X\twoheadrightarrow Y$ is a contraction if the inverse image of each 0-cell of $Y$ 
is a contractible subcomplex of $X$.  If $X$ is a $S\,\Delta$-complex then a set
surjection $f:S\twoheadrightarrow T$ induces a CW contraction on $X$ if $\overline{f^{-1}(t)}$ 
is contractible for each
$t\in T$.  In this case, we define $(X/f)$ to be the quotient of $X$ by the sub 
CW-complexes $\overline{f^{-1}(t)}$.  
The cocomposition map of $\textsc{cdc}$ takes 
$(T \xleftarrow{f} S)$ to the map which sends the 
$S\,\Delta$-complex $X$ to $(X/f)\otimes\bigl(\bigotimes_{t\in T} \overline{f^{-1}(t)}\bigr)$
if $f$ induces a CW contraction on $X$ and 0 otherwise.

\end{document}